\documentclass[11pt]{article}
\usepackage{amsmath,amsthm,amsfonts,amssymb,amscd, amsxtra, mathrsfs,textcomp,verbatim,inputenc}
\usepackage{url}
\usepackage{xcolor}
\usepackage{mathtools}
\usepackage{todonotes}
\usepackage{cancel}
\usepackage{booktabs}
\usepackage[shortlabels]{enumitem}
\usepackage{hyperref}

\usepackage{soul} 

\usepackage[margin=2.5 cm,nohead]{geometry}
\usepackage{algorithm}
\theoremstyle{plain}
\theoremstyle{definition}
\newtheorem{theorem}{Theorem}
\newtheorem{lemma}[theorem]{Lemma}
\newtheorem{definition}[theorem]{Definition}
\newtheorem{corollary}[theorem]{Corollary}
\newtheorem{proposition}[theorem]{Proposition}

\newtheorem{remark}{Remark}
\newtheorem{example}{Example}

\newtheorem{appendixthm}{Appendix}

\newtheorem*{properties*}{Properties of subdifferential}

\begin{document}

\title{A subdifferential characterization via Busemann functions and applications to DC optimization on Hadamard manifolds}
\author{
 O. P. Ferreira\thanks{Institute of Mathematics and Statistics, Federal University of Goias, Avenida Esperan\c{c}a, s/n, Campus II,  Goi\^ania, GO - 74690-900, Brazil (E-mail: {\tt orizon@ufg.br}, {\tt mauriciolouzeiro@ufg.br}).}
 \and
D. S. Gonçalves \thanks{Departamento de Matemática, Universidade Federal de Santa Catarina, Florianópolis, SC 88040-900, Brazil.
(E-mail:  {\tt douglas.goncalves@ufsc.br})}
\and
M. S. Louzeiro  \footnotemark[1]
\and
S. Z. N\'emeth \thanks{School of Mathematics, University of Birmingham, Watson Building, Edgbaston, Birmingham - B15 2TT, United Kingdom
(E-mails: {\tt s.nemeth@bham.ac.uk}, {\tt jxz755@bham.ac.uk}).}
\and
J. Zhu \footnotemark[3]
}

\maketitle
\begin{abstract}
This paper investigates the properties of Busemann functions on Hadamard manifolds and their use in optimization algorithms in Riemannian settings. We present a new Busemann-based characterization of the subdifferential, which is particularly well suited to Riemannian optimization. In the classical Hadamard manifold framework, a subgradient provides a global lower model of a convex function expressed through the inverse exponential map. However, this model may fail to exhibit a useful convexity or concavity structure. By contrast, our characterization yields a concave bounding function by exploiting key properties of Busemann functions. We  use  this concavity to design and analyze difference-of-convex (DC) optimization methods on Hadamard manifolds. In particular, we reformulate the classical DC algorithm (DCA)  for Riemannian contexts and study its  convergence properties. We also report preliminary numerical experiments comparing the proposed Busemann DCA, which leads to geodesically convex subproblems, with the classical Riemannian DCA.

\medskip

\end{abstract}

\section{Introduction}
The Busemann functions plays a crucial role in geometric topology, particularly in the analysis of manifolds like Hadamard manifolds. Introduced by H. Busemann in \cite{Busemann1955}, this concept captures the essence of the parallel axiom, facilitating the examination of geodesic geometry and offering insights into the behavior as they extend towards infinity. Its significance lies in its ability to reveal fundamental aspects of the overall structure of space. As an essential element in the study of Hadamard manifolds, the Busemann functions is expected to provide deeper insights as research progresses. To corroborate this, the discussion presented here demonstrates how this function can aid in the design and analysis of optimization algorithms for solving problems in Riemannian settings. For further exploration of the Busemann functions's role in geometry, refer to works such as \cite{Ballmann1985, BridsonHaefliger1999, Sakai1996}.

Given the well-established connections between Riemannian geometry and optimization, as evidenced by books  such as \cite{Absil2008, Boumal2020,
Rapcsak1997, Udriste1994}, and acknowledging the  fundamental  role of the Busemann functions in Riemannian geometry, it is natural to anticipate
its significance extending into the field of optimization. Indeed, the Busemann function have begun to attract attention from the optimization research community operating within the Hadamard manifolds  context. This attention has sparked efforts to delve deeper into Busemann functions, exploring their potential applications and theoretical implications across continuous optimization and machine learning domains. For instance, in \cite{Bento2022, BentoNeto2024}, the concept of a resolvent for equilibrium problems is introduced in terms of Busemann functions, alongside an investigation into elements of convex analysis in Hadamard manifold settings. Furthermore, \cite{Bento2023} presents a Fenchel-type conjugate utilizing Busemann functions, accompanied by an exploration of several properties. Additionally, \cite{hirai2023} discusses further developments in a
convex analysis foundation for  convex optimization on Hadamard spaces, incorporating the concept of the Legendre-Fenchel conjugate. In  \cite{lewis2024} a subgradient-type algorithm is proposed for exploring convex optimization on Hadamard spaces with the iteration based on the convexity of the sublevel sets of the Busemann functions.   Although immediate practical applications are not our primary focus, we note that \cite{bonet2023} introduces hyperbolic sliced-Wasserstein discrepancies using machine-learning tools. These discrepancies are constructed by projecting onto geodesics, for instance through projections defined by level sets of Busemann functions. This context also encourages further exploration of the Busemann functions, as discussed in another related work \cite{bonet2023-2}. Additionally, \cite{ghadimi2021} proposes a method termed hyperbolic Busemann learning for a classification problem, which aims to apply hierarchical relations among class labels to strategically position hyperbolic prototypes.

 In this paper, we introduce a novel characterization of the subdifferential of a convex function, a fundamental concept in nonsmooth optimization theory, based on Busemann functions. This characterization yields structural properties that are more suitable for optimization on Hadamard manifolds than those provided by the classical definition. Traditionally, in the Riemannian setting, the classical subgradient inequality guarantees that a geodesically convex function admits a global lower bound expressed through a subgradient and the inverse exponential map. In the Euclidean case, the corresponding lower model is affine and therefore is both convex and concave. However, on a general Hadamard manifold, the lower model induced by this classical definition may fail to be geodesically convex and may also fail to be geodesically concave. This notion, introduced in \cite[Definition~4.3, p.~73]{Udriste1994} and widely used in works such as \cite{FerreiraLouzeiroPrudente2018subgrad,LiMordukhovich2011,WangLiWangYao2015}, motivates the search for support constructions with a more favorable geometric structure. In contrast, our Busemann-based characterization yields a concave bounding function obtained from the negative of a Busemann function. This choice provides a convenient geometric structure for algorithmic design on Hadamard manifolds and suggests further directions for investigation beyond the scope of the present work. In particular, we use these Busemann-based bounds to address difference-of-convex (DC) optimization problems on Hadamard manifolds and to develop optimization methods suited to the Hadamard manifold framework. For instance, we explore our Busemann-based characterization to address the DC optimization problem on Hadamard manifolds. We examine the potential of the Busemann functions in optimization, highlighting its role in developing optimization methods designed to the Hadamard manifold framework. Specifically, we revisit the classic difference of convex algorithm (DCA) for Hadamard manifolds, initially introduced and analyzed in \cite{Bergmann2024} as the Hadamard manifold  counterpart to the celebrated Euclidean difference of convex algorithm (EDCA) introduced in \cite{Pham1986}, recently studied in \cite{Aragon2020}. However,  it is important to note  that the function involved in the subproblem of the classic DCA on Hadamard manifolds is generally \emph{not} convex, presenting significant solution-seeking challenges.  To overcome this limitation, we use the geometric structure provided by Busemann functions. Consequently, we propose a reformulation of the classical DCA in which the subproblem becomes geodesically convex, enabling a more effective treatment within the Riemannian setting.  
  
The paper is structured as follows. In Section~\ref{sec:int.01}, we establish some notations.  Section~\ref{sec:int.r} serves as a review of essential concepts, notations, and foundational results concerning Hadamard manifolds. In Section~\ref{sec:bfNewDef}, we delve into the Busemann functions on Hadamard manifolds, elucidating crucial properties, providing necessary notations, and offering illustrative examples.  Following this, in Section~\ref{sec:NewDefSub}, we introduce a novel characterization for subdifferential  based on the Busemann functions.  In Section~\ref{sec:OptProhm}, we present a Busemann-based algorithm for DC optimization on Hadamard manifolds, motivated by the use of Busemann supports as geometric  model  for linearization and we begin by revisiting the classical DCA in this setting.  In Section~\ref{sec:Numerics}, we present numerical experiments comparing the practical performance of the Busemann DCA with the classic Riemannian DCA. Finally, Section~\ref{sec:Conclusions} presents our conclusions.

\subsection{Notation} \label{sec:int.01}
 Let ${\mathbb R}^{m}$ be   the $m$-dimensional Euclidean space.  The set of all $m
 \times n$ matrices with real entries is denoted by ${\mathbb R}^{m \times n}$
 and ${\mathbb R}^m\equiv {\mathbb R}^{m\times 1}$.   For   $M \in {\mathbb
 R}^{m\times n}$ the matrix $M^{{T}}  \in {\mathbb R}^{n\times m}$  denotes  the
 {\it transpose} of $M$.     
 The matrix ${\rm I}$ denotes the $n\times n$ identity matrix. 
 Given $v \in \mathbb{R}^n$, ${\rm Diag}(v)$ denotes the $n \times n$ diagonal matrix with the entries of $v$ in its diagonal. 
 Denote by $\mathbb{R}^n_{++}$ the positive  orthant. 
 Let $\overline{\mathbb{R}}\coloneqq  \mathbb{R}\cup\{+\infty\}$ be the extended real line.   
 In line with  \cite{Aragon2020}, we will adopt the following conventions
\begin{align}
    \label{eq:conventions}
    (+\infty)-(+\infty) = +\infty,\quad
    (+\infty)-\lambda = +\infty, \text{ and }
    \lambda- (+\infty) = -\infty,
\end{align}
for all $\lambda \in \mathbb{R}$. The \emph{domain} of  $f:{\mathbb M} \to \overline{\mathbb{R}}$ is denoted  by ${\rm dom}\, f  \coloneqq \{ p\in {\mathbb M}:~f(p) < +\infty\}$.   Throughout this paper, we assume that ${\rm dom }\, f \neq \varnothing$, indicating  that  {\it $f$ is proper}.
\section{Basics results about Hadamard  manifolds} \label{sec:int.r}
In this section, we recall some concepts, notations, and basics results about Hadamard manifolds.
For more details see  for example, \cite{Manfredo1992, Sakai1996}.
{\it Throughout  this paper  ${\mathbb M}$ represents a finite dimensional Hadamard manifold} and
 $T_p{\mathbb M}$ the \emph{tangent space} of ${\mathbb M}$ at $p$.
The corresponding norm associated to the Riemannian metric $\langle \cdot , \cdot \rangle$
is {represented} by $\lVert\cdot\rVert$. We use $\ell(\gamma)$ to {express} the length of a piecewise smooth curve
$\gamma\colon [a,b] \rightarrow {\mathbb M}$. The Riemannian distance between $p$
and $q$ in ${\mathbb M}$ is  denoted by $d(p,q)$,
which induces the original topology on ${\mathbb M}$, namely, $({\mathbb M}, d)$,
which is a complete metric space.  The {\it exponential mapping}
$\exp_{p}:T_{p}{\mathbb M} \rightarrow  {\mathbb M} $ is defined  by $\exp_{p}v\,=\, \gamma _{p,v}(1)$, where $\gamma _{p,v}$ is the geodesic defined by its  initial position $p$, with velocity $v$ at $p$. Hence, we have $\gamma _{p,v}(t)=\exp_{p}(tv)$. Thus, {\it we will also use the expression $\exp_{p}(tv)$ for denoting  the geodesic   $\gamma_{{p}, v}$ starting  at $p\in {\mathbb M}$ with velocity $v\in T_p{\mathbb M}$ at $p$}.  For a $p\in{\mathbb M}$, the exponential map $\exp_p$ is a diffeomorphism and $\log_p\colon{\mathbb M}\to T_p{\mathbb M}$ {indicates} its inverse.  In this case, $d(p,q) = \|\log_pq\|$ holds,  $d(\cdot, q)\colon{\mathbb M}\backslash \{q\}\to\mathbb{R}$  is $C^{\infty}$  for all $q\in {\mathbb M}$ and  its gradient is given by ${\rm grad}_{1} d(p, q) = (-\log_pq)/d(q, p)$, for all $q\neq p$, where ${\rm grad}_{1}$ denotes the gradient with respect to first coordinate. In addition,  $d^2(\cdot, q)\colon{\mathbb M}\to\mathbb{R}$  is $C^{\infty}$  for all $q\in {\mathbb M}$, and  ${\rm grad}_{1} d^2(p, q) = -2\log_pq$.   Let $\bar{p},\bar{q}\in {\mathbb M}$ and $(p_{k})_{k\in \mathbb{N}}, (q_{k})_{k\in \mathbb{N}}\subset {\mathbb M}$ be sequences  such that $\lim_{k\to +\infty} p_k=\bar{p}$ and $\lim_{k\to +\infty}q_k=\bar{q}$. Then, for any $q\in \mathbb{M}$,  $\lim_{k\to +\infty} \log_{p_k}q=  \log_{\bar{p}}q$ and $\lim_{k\to +\infty} \log_qp_k= \log_q\bar{p}$ and   $\lim_{k\to +\infty}  \log_{p_k}q_k=\log_{\bar{p}}\bar{q}$.  Given $p,q\in{\mathbb M}$,  {the symbol $\gamma_{pq}$ indicates} the geodesic segment  joining  $p$ to $q$, i.e., $\gamma_{pq}\colon[0,1]\rightarrow{\mathbb M}$ with $\gamma_{pq}(0)=p$ and $\gamma_{pq}(1)=q$.    In the following, we will recall the well-known ``comparison theorem" for triangles in Hadamard manifolds, as stated in \cite[Proposition 4.5]{Sakai1996}. 
 \begin{lemma}  \label{le:CosLawF}
Let ${\mathbb M}$ be a  Hadamard  manifold. The following inequality  holds:
 \begin{equation} \label{eq:coslaw2}
d^2({x}, {y})+d^2({x},{z})-2\left\langle  \log_{{x}}{y}, \log_{{x}}{z}\right\rangle\leq d^2({y},{z}),  \qquad   \forall {x}, {y}, {z} \in {\mathbb M}. 
\end{equation}
As a consequence,
\begin{equation} \label{eq:log-norm-ineq}
\left\| \log_{x} y - \log_{x} z \right\| \leq d(y, z), \qquad \forall x, y, z \in {\mathbb M}.
\end{equation}
Moreover, if the sectional curvature of ${\mathbb M}$ is identically zero, then both inequalities \eqref{eq:coslaw2} and \eqref{eq:log-norm-ineq} hold as equalities.
\end{lemma}

\begin{definition} \label{def:Lip}
Let $ \mathbb{M} $ be a Hadamard manifold, and let $ f \colon \mathbb{M} \rightarrow \overline{\mathbb{R}} $ be a function. 
We say that $ f $ is \emph{$ L $-Lipschitz} on a subset $ \Omega \subset \mathbb{M} $, for some constant $ L \geq 0 $, if $|f(p) - f(q)| \leq L\, d(p,q), $
for all  $p, q \in \Omega \cap \operatorname{dom}\, f$, where $ d $ denotes the Riemannian distance on $ \mathbb{M} $.
\end{definition}
The following definition plays an important role in the paper (see \cite[p. 363]{Bourbaki1995}).

\begin{definition} \label{eq:lscf}
A function $f\colon {\mathbb M} \to \overline{\mathbb{R}}$ is said to be
\emph{lower semi-continuous} (\emph{lsc}) at a point $p \in {\mathbb M}$ if $\liminf_{q \to p} f(q) =f(p)$. If $f$ is lower semi-continuous at every point in ${\mathbb M}$, we simply say that $f$ is \emph{lower semi-continuous}.
\end{definition} 
We conclude this section with a useful property of lower semi-continuous functions on Hadamard manifolds, as stated in the next proposition. Its proof is similar to that of the Euclidean space and will be omitted here.
\begin{proposition}\label{prop:CoeFunc} 
	Let ${\mathbb M}$ be a Hadamard manifold  and
	$f\colon{\mathbb M}\to \overline{\mathbb{R}}$ be a lower semi-continuous function.  If $f$ is coercive, i.e.,  $ \lim_{d({p},{\bar p})\to+\infty } {f(p)} = +\infty$, for some fixed ${\bar p}\in {\mathbb M}$, then  $f$ has a global minimizer in ${\mathbb M}$.
\end{proposition}

Perhaps the two most important examples of Hadamard manifolds in optimization applications, apart from the Euclidean space $\mathbb{R}^n$, are the $\kappa$-hyperbolic space form ${\mathbb H}^{n}_{\kappa}$ and the space of symmetric positive definite matrices $\mathcal{P}(n)$. In the next two sections, we provide a brief review of their key properties, which will serve as the foundation for the examples and numerical experiments developed in this work.
\subsection{Basic results on the $\kappa$-hyperbolic space form} \label{sec:int.1}

In this section, we provide a review of the  basic results related to the geometry of the ${\kappa}$-hyperbolic space forms. References to this section include  \cite{BenedettiPetronio1992, Boumal2020, Ratcliffe2019}. For a given $\kappa>0$,  the {\it $n$-dimensional  ${\kappa}$-hyperbolic space form} and its {\it tangent hyperplane at a point $p$} are denoted  by
\begin{equation*}
{\mathbb H}^{n}_{\kappa}:=\Big\{ p\in {{\mathbb R}^{n+1}}:~\langle p, p\rangle=-\frac{1}{\kappa}, ~p_{n+1}>0\Big\}, \qquad T_{p}{{\mathbb H}^n_{\kappa}}:=\left\{v\in {{\mathbb R}^{n+1}}\, :\, 
\left\langle p, v \right\rangle=0\right\},
\end{equation*}
where,  $\left\langle \cdot , \cdot \right\rangle$ is the   {\it Lorentzian inner product}   $\left\langle x, y\right\rangle:= x^{{T}}{\rm J}y$ and ${\rm J}:={\rm Diag}(1, \ldots,1, -1) \in {\mathbb R}^{(n+1)\times (n+1)}.$    The  {\it  Lorentzian projection into} $T_p{{\mathbb H}^n_{\kappa}}$ is the linear mapping ${\rm Proj}^{\kappa}_p:{{\mathbb R}^{n+1}} \to T_p{{\mathbb H}^n_{\kappa}}$  defined by
$
 {\rm Proj}^{\kappa}_p x:=x+ \kappa \langle p, x \rangle p,
$
 i.e., $ {\rm Proj}^{\kappa}_p:={\rm I} +\kappa pp^{{T}}{\rm J} $, where ${\rm I}  \in
 {\mathbb R}^{(n+1)\times (n+1)}$ is the identity matrix. The {\it intrinsic distance on the  $\kappa$-hyperbolic space form} between two  points $p, q \in {{\mathbb H}^n_{\kappa}}$  is  given~by
\begin{equation} \label{eq:Intdist}
d_{\kappa}(p, q):=\frac{1}{\sqrt{\kappa}}{\rm arcosh} (-\kappa \left\langle p , q\right\rangle).
\end{equation}
The {\it exponential mapping} $\exp^{\kappa}_{q}:T_{q}{{\mathbb H}^n_{\kappa}} \rightarrow {{\mathbb H}^n_{\kappa}}$ is given by  $\exp^{\kappa}_{q}v=q$ for  $v=0$,  and 
\begin{equation*} 
\exp^{\kappa}_{q}v:= \displaystyle \cosh({\sqrt{\kappa}}\|v\|) \,q+ \sinh(\sqrt{\kappa}\|v\|)\, \frac{v}{{\sqrt{\kappa}}\|v\|},
\qquad  \forall v\in T_q{{\mathbb H}^n_{\kappa}}\setminus\{0\}.
\end{equation*} 
The {\it inverse of the exponential mapping}  $\log^{\kappa}_{q}:{{\mathbb H}^n_{\kappa}} \to T_{q}{{\mathbb H}^n_{\kappa}}$ at ${q\in {\mathbb H}^n_{\kappa}} $ is given by $\log^{\kappa}_{q}p=0$, for $p=q$, and 
\begin{equation*} 
\log^{\kappa}_{q}p:=  \displaystyle \frac{\sqrt{\kappa}d_{\kappa}(q, p)}{\sqrt{{\kappa}^2\left\langle q, p\right\rangle^2-1}} 
{\rm Proj}^{\kappa}_qp= \displaystyle  d_{\kappa}(q, p)\frac{{\rm Proj}^{\kappa}_qp}{\|{\rm Proj}^{\kappa}_qp\|},  \qquad  p\neq q.
\end{equation*}
Let  $\Omega \subset {\mathbb H}^{n}_{\kappa}$ be an open set.  The {\it Riemannian gradient on the ${\kappa}$-hyperbolic space form} of a differentiable function   $f: \Omega \to  {\mathbb R}$  is the unique vector field $\Omega \ni p\mapsto {\rm grad} f(p)\in T_{p}{\mathbb M}$ such that $df(p)v=\left\langle {\rm grad} f(p), v\right\rangle$, see \cite[Proposition~7-5, p.162]{Boumal2020}.  Therefore, we have
\begin{equation} \label{eq:grad}
{\rm grad} f(p):= {\rm Proj}^{\kappa}_p {\rm J}f'(p)=  {\rm J}f'(p)+ \kappa \left\langle {\rm J}f'(p), p\right\rangle \,p,
\end{equation}
where $f'(p) \in {\mathbb R}^{n+1}$ is the usual gradient of  $ f$ at $p$.

\subsection{Basic results on the manifold of symmetric positive definite matrices} \label{sec:int.2}
In this section, we provide a review of the  basic results related to the geometry of the manifold of symmetric positive definite matrices. References to this section include  \cite{Boumal2020, BridsonHaefliger1999}. Let $ \mathbb{R}^{n\times m} $ denote the space of real matrices of size $ n \times m $, $ \mathcal{S}(n) \subset \mathbb{R}^{n\times n} $ the set of symmetric matrices, and $ \mathcal{P}(n) \subset \mathbb{R}^{n\times n} $ the set of symmetric positive definite matrices. We equip $\mathcal{P}(n)$ with the \emph{affine-invariant Riemannian metric}, which is defined by
\begin{equation} \label{eq:AffineInvariantMetric}
    \langle U, V \rangle := {\rm tr}\left(Y^{-1} U Y^{-1} V\right), \quad \forall Y \in \mathcal{P}(n), \quad U, V \in T_Y \mathcal{P}(n) \equiv \mathcal{S}(n),
\end{equation}
where $ {\rm tr}(\cdot) $ denotes the trace operator, and $T_Y \mathcal{P}(n)$ is the tangent space at $Y$, identified with $ \mathcal{S}(n) $. This metric endows $ \mathcal{P}(n) $ with the structure of a Hadamard manifold.   The {\it affine-invariant Riemannian distance} between two points $X, Y \in \mathcal{P}(n)$ is defined by
    \begin{equation} \label{eq:RiemannianDistance}
        d(X, Y) =  {\rm tr}^{1/2}\left(  {\rm Log}^2\left(Y^{-1/2} X Y^{-1/2}\right)  \right)  = {\rm tr}^{1/2}\left(  {\rm Log}^2\left(Y^{-1} X \right)  \right) ,
    \end{equation}
    where ${\rm Log}$ denotes the usual matrix logarithm. The  exponential map at $Y \in \mathcal{P}(n)$ with respect to the affine-invariant metric given by
    \begin{equation} \label{eq:ExpSPD}
        \exp_Y(V) := Y^{1/2} {\rm Exp}\left(Y^{-1/2} V Y^{-1/2}\right) Y^{1/2},
    \end{equation}
    for all $V \in T_Y \mathcal{P}(n)$, where  ${\rm Exp}$ denotes the  usual matrix exponential. The inverse of the exponential map, is denoted by 
  \begin{equation} \label{eq:LogSPD}
\log_XY :=X^{1/2}{\rm Log}(X^{-1/2} Y X^{-1/2}) X^{1/2}.
    \end{equation}
 The {\it gradient} on the manifold of symmetric positive definite matrices  $\mathcal{P}(n)$  of a differentiable function   $f: \mathcal{P}(n) \to  {\mathbb R}$  is the unique vector field $\mathcal{P}(n) \ni X\mapsto {\rm grad} f(X)\in \mathcal{S}(n)$  given by 
      \begin{equation} \label{eq:gradSPD}
    \mbox{grad} f(X)=Xf'(X)X, 
   \end{equation}
  where $f'(X) \in \mathcal{S}(n)$ denotes the Euclidean gradient of  $ f$ at $X$.

We conclude this section with a useful property of the Riemannian distance that is instrumental in computing Busemann functions on the manifold of symmetric positive definite matrices. Its proof follows directly from \eqref{eq:RiemannianDistance} and is therefore omitted.
  \begin{lemma}\label{lem:Log.property}
 If $ X \in \mathcal{P}(n) $  and $ V $ is a non-singular matrix, then ${\rm Log}(V X V^{-1}) = V ({\rm Log} X) V^{-1}.$ As a consequence,  for given  $ X, Y \in \mathcal{P}(n)$ and  non-singular matrices  $Z$ and $ V $ such that   $ ZXV $ and $ Z^{-1} Y V^{-1} $ both belong to $ \mathcal{P}(n) $, there holds 
$
d\left( ZXV, Y \right) = d\left( X, Z^{-1} Y V^{-1} \right).
$
\end{lemma}

\subsection{Subdifferential of a convex function   on Hadamard manifolds}

In this section, we recall the subdifferential of a convex function on a Hadamard manifold and its main properties, as presented in \cite{Udriste1994} and further developed in \cite{Ferreira2002, LiMordukhovich2011}. These classical notions are reviewed to prepare the discussion on the characterization of the subdifferential via Busemann function  in the subsequent sections. We begin with the definitions of convex sets and convex functions.

A set $\Omega \subset {\mathbb M}$ is said to be convex, if for all $p,q\in \Omega$ we have $\gamma_{pq}(t)\in \Omega$, for all $t\in [0,1]$.  A function $f\colon{\mathbb M} \to \overline{\mathbb{R}}$ is \emph{$\sigma$-strongly convex} for $\sigma \geq 0$ if $(f\circ\gamma_{pq})(t)\le (1-t)f(p)+tf(q)-\frac{\sigma}{2}t(1-t)d^2(p,q)$,  for all $p,q\in{\mathbb M}$ and $t\in[0,1]$.  In particular, $f$ is \emph{convex} when $\sigma=0$. For $\sigma=0$, $f$ is \emph{strictly convex} if the inequality is strict for all $p\neq q$ in ${\rm dom}\, f$ and all $t\in(0,1)$.

\begin{definition} \label{def:Sugd}	
Let  $f\colon \mathbb M \rightarrow \overline{\mathbb{R}}$  be a convex function and $q\in {\rm dom}\, f$.  A vector $ s\in T_{{q}}{\mathbb M}$ is said to be subgradient  of $f$ at ${{q}}$  if
\begin{equation} \label{eq:Sugd}
f({p}) \geq f({{q}})+\left\langle  s , \log_{ {q} } p\right\rangle, \qquad \forall p\in {\mathbb M}.
\end{equation}
The set of all subgradients of  $f$ at the point ${q}$ is called the subdifferential and is denoted by $\partial f({q})$.
\end{definition}
 It is a well-established fact that the subdifferential set $\partial f({p})$ is nonempty for every ${p}\in {\rm int} \, {\rm dom} \, f$. For an analytic proof, refer to \cite[Theorem~4.5]{Udriste1994}, and for a geometric proof, see \cite{Ferreira2002}. Moreover, $\partial f({p})$ is recognized as a convex and compact set, as demonstrated in \cite[Theorem 4.6]{Udriste1994}. To explore further useful properties of $\partial f(p)$, we denote by $f'(p, v)$ the {\it directional derivative of $f$ at $p$ in the direction of $v \in T_p \mathbb{M}$}, as defined in \cite[Definition 4.1]{Udriste1994}.  Recall that, for a  given $p\in \mathbb M$,  we have 
 $
{\rm dom}\, f'(p, \cdot):=\bigl\{\,v\in T_p\mathbb{M}:~\exists\ {\hat t}>0; \exp_p(tv)\in {\rm dom }\, f\ \ \forall\,t\in[0,  {\hat t})\bigr\}.
$
The proof of the first part of  the next  result can be found  in \cite[Theorem 4.8]{Udriste1994}; for additional details, see \cite{NetoFerreiraLucambio2000} and \cite[Proposition 3.8(ii)]{LiMordukhovich2011}. The second part is addressed in \cite[Proposition 4.3]{LiMordukhovich2011}.
\begin{proposition}\label{pr:sfsd}
Let  $f\colon \mathbb M \rightarrow \overline{\mathbb{R}}$ be a convex function. Then,    for each fixed $p\in {\rm dom } \, f$,  there holds  $\partial f(p)=\left\{s \in T_{p}{\mathbb M} :~ f'(p,v) \geq \langle s,v \rangle,  v \in T_{p}{\mathbb M}\right\}$. In addition, if $g\colon \mathbb M \rightarrow \overline{\mathbb{R}}$ is   a convex function such that ${\rm dom } \, f\cap {\rm dom }\, g$ is convex, then  $\partial (f+g)(p)=\partial f(p)+ \partial g(p)$, for  each $p\in ({\rm int} \,{\rm dom }\, f)\cap {\rm dom } \, g$.
\end{proposition}
The proof of the first claim in the theorem below can be found in \cite[Theorem 4.10, p. 76]{Udriste1994}, with the proof of the second claim following a similar approach.
\begin{theorem} \label{thm:f-convex-subdiff}
Let  $f\colon \mathbb M \rightarrow \overline{\mathbb{R}}$ be a function. Then $f$ is convex (resp. $\sigma$-strongly convex) if and only if ${\rm dom} \, f$ is convex and, for every $p\in{\rm dom}\, f$, there exists $v\in T_p\mathbb M$ such that
$f(q)\ge f(p)+\langle v,\log_p q \rangle$ (resp.  $f(q)\ge f(p)+\langle v,\log_p q\rangle+\tfrac{\sigma}{2}d^2(p,q))$, for all $q\in\mathbb M$. In either case, $\partial f(p)\neq\emptyset$ for all $p\in{\rm dom} \, f$, and the inequality holds for every $v\in\partial f(p)$.
\end{theorem}
The proof of the following result immediately follows from  \cite[Proposition 2.5]{WangLiWangYao2015}.
\begin{proposition}
    \label{cont_subdif}
    Let $f\colon \mathbb M \rightarrow \overline{\mathbb{R}}$ be a convex
    and lower semicontinuous function. Consider the sequence
    $(p_{k})_{k\in\mathbb{N}}\subset \mbox{int} \, {\rm dom } \, f$ such that $\lim_{k\to\infty}p_{k}={\bar p} \in \mbox{int} \, {\rm dom } \, f$.
    If $(v_{k})_{k\in\mathbb{N}}$ is a sequence such that $v_{k}\in \partial f(p_{k})$
    for every $k\in \mathbb{N}$, then $(v_{k})_{k\in\mathbb{N}}$ is bounded and
    its cluster points belong to $\partial f({\bar p}).$
\end{proposition}
\section{The Busemann functions on Hadamard manifolds} \label{sec:bfNewDef}

In this section, we review Busemann functions on Hadamard manifolds, introducing the notation and collecting the properties needed in the sequel. Since our definition is slightly more general than the standard one, we include brief proofs of selected results. For further background, see, for instance, \cite{Sakai1996}.

Let $\mathbb{M}$ be a Hadamard manifold with Riemannian distance $d$. Given a base point $q \in \mathbb{M}$ and a vector $v \in T_q \mathbb{M}$, the associated \emph{Busemann function} is defined by
\begin{equation} \label{eq:defBFNDef}
B_{{q}, v}(p) \coloneqq \lim_{t \to+ \infty} \left(d\left({p}, \exp_{q}(tv)\right)-\|v\| t\right),     \qquad  \forall p\in {\mathbb M}. 
\end{equation}
For $v=0$, this reduces to $B_{q,0}(p)=d(q,p)$. Moreover, by the triangle inequality,
\begin{equation} \label{eq:defFIWDef2b}
|B_{{q}, v}(p)|\leq d({q}, p),    \qquad  \forall q, p \in {\mathbb M}, \quad \forall v\in T_{q}{\mathbb M}.
\end{equation}
\begin{remark} \label{re:bfn1}
Classically, Busemann functions are defined for unit vectors $v \neq 0$; see, e.g.,  \cite[Definition 8.17, p. 268]{BridsonHaefliger1999} and  \cite[p. 174]{Sakai1996}. Since $B_{q,v} = B_{q,v/\|v\|}$ for $v \neq 0$, we extend the definition to arbitrary $v$, including $v=0$, which is convenient for our purposes.
\end{remark}

The following lemma summarizes the main regularity properties of Busemann functions and provides two equivalent expressions for their gradients.

\begin{lemma} \label{le:CharactBusFunc}
Let $q \in \mathbb{M}$ and $v \in T_q\mathbb{M}$ with $v \neq 0$. Then $B_{q,v}$ is convex and continuously differentiable on $\mathbb{M}$, with
\begin{align}
{\rm grad}\, B_{q,v}(p)
&= -\lim_{t \to \infty} \frac{\log_p(\exp_q(tv))}{d(p,\exp_q(tv))} \label{eq:gradbf1}\\
&= -\frac{1}{\|v\|} \lim_{t \to \infty} \frac{\log_p(\exp_q(tv))}{t}, 
\qquad \forall p \in \mathbb{M}. \label{eq:gradbf2}
\end{align}
Moreover, $\|{\rm grad}\, B_{q,v}(p)\|=1$ for all $p \in \mathbb{M}$, and
$
{\rm grad}\, B_{q,v}(q) = -v/\|v\|.
$
\end{lemma}
\begin{proof}
The identity in \eqref{eq:gradbf1} is established in the proof of \cite[Lemma~4.12, p.~231]{Sakai1996}. To prove \eqref{eq:gradbf2}, note that by the triangle inequality we have
$t\|v\| - d(p,q) \leq d(p, \exp_{q}(tv)) \leq t\|v\| + d(p,q)$ for all $t>0$. Thus, dividing by $t$ and then taking the limit as $t$ goes to $+ \infty$, we conclude that
$
\lim_{t\to \infty} ({d(p, \exp_{q}(tv))}/{t}) = \|v\|.
$
Combining this limit with \eqref{eq:gradbf1}, we obtain
\begin{equation*}
{\rm grad} B_{q, v}(p)
= -\frac{1}{\|v\|}  \lim_{t \to \infty} \frac{\log_{p}(\exp_{q}(tv))}{d(p, \exp_{q}(tv))} \lim_{t \to \infty} \frac{d(p, \exp_{q}(tv))}{t}  
=- \frac{1}{\|v\|} \lim_{t \to \infty} \frac{\log_{p}(\exp_{q}(tv))}{t},
\end{equation*}
which completes the proof.
\end{proof}

Note that, for $v=0$, we have $B_{{q}, 0}=d({q}, p)$. Therefore,   we also obtain that  $B_{{q}, 0}$ is  convex  and continuously  differentiable with  the gradient vector field ${\rm grad} B_{{q}, 0}$ satisfying  $\|{\rm grad} B_{{q}, 0}(p)\|=1$, for all $p\neq {q}$.  The following lemma, whose proof is straightforward and thus omitted, in particular implies that the Busemann functions $B_{{q}, v}$ is linear along the geodesic $t \mapsto \exp_q(tv)$ that defines it.
\begin{lemma} \label{eq:pbfu}
Let ${q} \in {\mathbb M}$ be a base point and ${w} \in T_q {\mathbb M}$. Then,
$
B_{{q},v}(\exp_{{q}}(\tau v)) = \tau \|v\|
$
 for all $\tau \in {\mathbb R}$. Consequently, the Busemann functions $B_{{q},v}$ is unbounded both above and below.
\end{lemma}
The following lemma records a continuity property of the Busemann functions that will be particularly useful in Section~\ref{sec:OptProhm}. A proof can be obtained by adapting the argument of \cite[Lemma~5]{Criscitiello2025}, see also \cite[Chapter II.1]{Ballmann1995}.
\begin{lemma}\label{le:cbf} 
Let $\bar{q} \in \mathbb{M}$ and $\bar{w} \in T_{\bar{q}} \mathbb{M}$. Consider sequences $(q_k)_{k \in \mathbb{N}} \subset \mathbb{M}$ and $(v_k)_{k \in \mathbb{N}}$ with $v_k \in T_{q_k} \mathbb{M}$, satisfying  $\lim_{k \to +\infty} q_k = \bar{q}$, and $\lim_{k \to +\infty} v_k = \bar{w}$. Then, 
$\lim_{k \to +\infty} B_{q_k, v_k}(p) = B_{\bar{q}, \bar{w}}(p),$ for all $p \in \mathbb{M}$.
\end{lemma}
We conclude this section with a useful identity, obtained directly from \eqref{eq:defBFNDef}, which provides a practical means for computing the Busemann functions.
\begin{proposition}  \label{pr:afcbf}
Let ${q} \in  {\mathbb M}$ be a base point and    ${w} \in T_{q} {\mathbb M}$ with  $v\neq 0$.  Then, there holds
\begin{equation*} 
B_{{q}, v}(p) = \lim_{t \to+ \infty} \frac{(d^2\left({p}, \exp_{q}(tv))-(\|v\|t)^2\right)}{2\|v\|t},     \qquad  \forall p\in {\mathbb M}. 
\end{equation*}
\end{proposition} 

\subsection{Examples  of  Busemann functions} 
In this section, we provide examples of the Busemann functions. We begin by introducing a fundamental property that establishes a stronger inequality than \eqref{eq:defFIWDef2b} in the case where the Busemann functions is positive. This property not only serves as motivation for the examples presented in Hadamard manifolds with identically zero sectional curvature but also plays an essential role in subsequent sections.

\begin{lemma} \label{le:busfunc}
Let ${\mathbb M}$ be a Hadamard manifold. Then the Busemann functions $B_{q, v}$ defined in \eqref{eq:defBFNDef}, associated with a base point $q \in {\mathbb M}$ and a direction $v \in T_q{\mathbb M}$, satisfies the inequality
\begin{equation} \label{eq:mainineq} 
-\left\langle v, \log_q p \right\rangle\leq \|v\|  B_{q, v}(p), \qquad \forall p \in {\mathbb M}.
\end{equation}
Moreover, if  the sectional curvature  is $K\equiv 0$ in whole ${\mathbb M}$, then inequality \eqref{eq:mainineq}  holds as equality
$\|v\| \, B_{q, v}(p) = -\left\langle v, \log_q p \right\rangle$, for all  $p \in {\mathbb M}$.
\end{lemma}
\begin{proof}
It is immediate that \eqref{eq:mainineq} holds for \(v=0\). We now assume \(v\neq 0\). By applying Lemma~\ref{le:CosLawF} to the geodesic triangle with vertices \(x=q\), \(y=p\), and \(z=\exp_q(tv)\), it follows that
 \begin{equation} \label{eq:coslawN}
d^2({q},{p})+ d^2({q}, {\exp_{q}(tv)})-2\left\langle \log_{{q}}{p}, \log_{{q}}{\exp_{q}(tv)}\right\rangle\leq d^2({p},{\exp_{q}(tv)}), \quad \forall t>0.
\end{equation}
Since  $\log_{{q}}{\exp_{q}(tv)}=tv$ and $d({q}, {\exp_{q}(tv)})=t\|v\|$, it  follows from the last inequality that 
\begin{equation*} 
d^2({q},{p})-2t\left\langle \log_{{q}}{p}, v\right\rangle\leq (d^2({p},{\exp_{q}(tv)})-(\|v\| t)^2), \quad \forall t>0.
\end{equation*}
After performing some  algebraic manipulations, the last inequality can be expressed as follows
$$
\frac{d^2({q},{p})}{2t}-\left\langle \log_{{q}}{p}, v\right\rangle\leq \|v\| \frac{(d^2\left({p}, \exp_{q}(tv))-(\|v\|t)^2\right)}{2\|v\|t},  \quad \forall t>0.
$$
Taking the limit in the previous inequality as \(t\to +\infty\) and using Proposition~\ref{pr:afcbf}, we obtain \eqref{eq:mainineq}. Moreover, by Lemma~\ref{le:CosLawF}, if the sectional curvature satisfies \(K\equiv 0\) on \(\mathbb{M}\), then \eqref{eq:coslawN} holds with equality. Consequently, all subsequent inequalities are equalities, which proves the desired equality.
\end{proof}
In the following example, we present an explicit formula for Busemann functions on a Hadamard manifold with identically zero sectional curvature.
\begin{example}\label{ex:defBFNDefK0}
Let \({\mathbb M}\) be a Hadamard manifold, \(q \in {\mathbb M}\) a base point, and \(v \in T_q{\mathbb M}\). If the sectional curvature of \({\mathbb M}\) is identically zero, denoted by \(K \equiv 0\), then the Busemann function \(B_{q,v}\) is given by
\begin{equation}\label{eq:defBFNDefK0}
B_{q,v}(p) \coloneqq
\begin{cases}
\displaystyle \Big\langle -\frac{v}{\|v\|},\, \log_q p \Big\rangle, 
& \forall\, v \in T_q{\mathbb M},~ v \neq 0, \\[1ex]
d(q,p), 
& v = 0.
\end{cases}
\end{equation}
Indeed, when \(K \equiv 0\) on \({\mathbb M}\), Lemma~\ref{le:busfunc} yields
\[
\|v\|\, B_{q,v}(p) = -\langle v, \log_q p \rangle,
\qquad \forall\, p \in {\mathbb M},
\]
which, together with \(B_{q,0}(p)=d(q,p)\), implies \eqref{eq:defBFNDefK0}. 
In particular, for \({\mathbb M} = {\mathbb R}^n\), we obtain
\begin{equation}\label{eq:defBFNDefEucl}
B_{q,v}(p) \coloneqq
\begin{cases}
\displaystyle -\Big\langle \frac{v}{\|v\|},\, p-q \Big\rangle, 
& \forall\, v \in T_q{\mathbb M},~ v \neq 0, \\[1ex]
\|q-p\|, 
& v = 0.
\end{cases}
\end{equation}
Since \(\log_q p = p-q\) and \(d(q,p)=\|p-q\|\) in \({\mathbb R}^n\), \eqref{eq:defBFNDefEucl} follows directly from \eqref{eq:defBFNDefK0}.
\end{example}
Next, we use \eqref{eq:defBFNDefK0} to derive an explicit expression for the Busemann functions on the positive orthant endowed with the Dikin metric. A detailed study of the positive orthant with the Dikin metric is given in \cite{Todd2002}.
\begin{example}
Let ${\mathbb M}\coloneqq (\mathbb{R}_{++}^n,G)$  be the positive orthant  $\mathbb{R}^n_{++}$ endowed  with the  Dikin metric    $\langle u,v \rangle:=u^TG(q)v$, where $u, v\in T_{q} {\mathbb M}$  and   $G(q)\in {\mathbb R}^{n \times n}$ is  the   diagonal matrix 
$$
G(q) \coloneqq {\rm diag}\big(q_{1}^{-2},\dots,q_{n}^{-2}\big), 
$$
and $q_{i}$ denotes the  $i$-th  coordinate of the point $q$.  The exponential map $\exp_{q}\colon T_{q} {\mathbb M} \to {\mathbb M}$  is given by
\begin{equation*} 
\exp_{q}(v) = \big({q}_1e^{{v_1}/{{q}_1}},\ldots,{q}_ne^{{v_n}/{{q}_n}}\big), \qquad v:=(v_1,\ldots,v_n)\in T_{q}{\mathbb M}\equiv\mathbb{R}^n.
\end{equation*}
In addition, direct calculations show that the inverse of the exponential $\log_q:{\mathbb M}\to T_{q}{\mathbb M}$ is given by 
\begin{equation*}
\log_{q} {p} = \left(q_1\ln({p_1}/{q_1}),\ldots,q_n\ln({p_n}/{q_n})\right), \qquad p:=(p_1,\ldots,p_n)\in {\mathbb M}.
\end{equation*}
Therefore, it follows from Example~\ref{ex:defBFNDefK0} that  the  Busemann functions $B_{{q}, v}$  is given by 
\begin{equation*} 
B_{{q}, v}(p):= \begin{cases}
-\frac{1}{\|v\|}\sum_{i=1}^n({v_i}/{q_i})\ln({p_i}/{{q_i}}),   &   \qquad  \forall v\in T_{q}{\mathbb M},~ v\neq 0, \\
\left(\sum_{i=1}^n \bigl[\ln(p_i/q_i)\bigr]^2\right)^{1/2},  &  \qquad v=0.
\end{cases}
\end{equation*}
\end{example}
In the following example, we provide an explicit formula for the Busemann functions on the $\kappa$-hyperbolic space form ${\mathbb H}^{n}_{\kappa}$. For the detailed computations, see Appendix~\ref{app:A3}.

\begin{example}  \label{ex:BusHype}
	Let ${\mathbb H}^{n}_{\kappa}$ be  the $\kappa$-hyperbolic space form introduced in  Section~\ref{sec:int.1}. Let ${q} \in  {\mathbb H}^{n}_{\kappa}$ be  a base point,  ${w} \in T_{q}{\mathbb H}^{n}_{\kappa}$ .  Then, the  Busemann functions $B_{{q}, v}:   {\mathbb H}^{n}_{\kappa} \to {\mathbb R}$  is given by 
\begin{equation} \label{eq:BFkSF}
B_{{q}, v}(p):= \begin{cases}
 \frac{1}{\sqrt{\kappa}}\ln\left( - \left\langle p, \kappa\,{q}+ {\sqrt{\kappa}}\,\frac{1}{\|v\|}{v}\right\rangle \right),  &   \qquad  \forall v\in T_{q}{\mathbb M},~ v\neq 0, \\
d({q},p),  &  \qquad v=0.
\end{cases}
\end{equation}
and its gradient, for $v\in T_{q}{\mathbb M}$ with  $v\neq 0$ is given by 
\begin{equation} \label{eq:grad2}
	{\rm grad} B_{{q}, v}(p)=\frac{1}{\sqrt{\kappa}}\frac{1}{\left\langle p, \kappa\,{q}+ {\sqrt{\kappa}}\frac{1}{\|v\|}{v}\right\rangle}\left(\kappa\,{q}+ {\sqrt{\kappa}}\frac{1}{\|v\|}{v}+ \kappa \left\langle \kappa\,{q}+ {\sqrt{\kappa}}\frac{1}{\|v\|}{v}, p\right\rangle \,p \right).
\end{equation}
Particularly, if $p=q$, then due to  $\langle q, q\rangle=-\frac{1}{\kappa}$ and
$\langle p, v \rangle=0$, the final equation simplifies to
\begin{equation*}
{\rm grad} B_{{q}, v}(q)= -\frac{1}{\sqrt{\kappa}}\left(\kappa\,{q}+ {\sqrt{\kappa}}\frac{1}{\|v\|}{v}- \kappa q \right)=-\frac{1}{\|v\|}{v}.
\end{equation*}
\end{example}
In the next example, we derive explicit expressions for the Busemann functions and their Riemannian gradients on the manifold of symmetric positive definite matrices described in Section~\ref{sec:int.2}. Although alternative formulas are available in the literature \cite[Proposition~10.69]{BridsonHaefliger1999}, \cite[Lemma~2.32]{hirai2023}, we introduce a new representation that is computationally cheaper and better suited for numerical applications. For completeness, Appendix~\ref{app:A4} provides a direct derivation that avoids the general theory of symmetric spaces used in previous approaches.
\begin{example}  \label{ex:BusSDP}
Let $ \mathcal{P}(n) $ be endowed with the structure of a Hadamard manifold as introduced in Section~\ref{sec:int.2}. Given $ X, Y \in \mathcal{P}(n) $ and $ V \in \mathcal{S}(n) \setminus \{0\} $, consider the spectral decomposition
\begin{equation*}\label{decom.V}
Y^{-1/2} V Y^{-1/2} = U D U^{T}, \qquad 
D := \begin{bmatrix}
\lambda_1 I_{n_1} & 0 & \cdots & 0\\
0 & \lambda_2 I_{n_2} & \cdots & 0 \\
\vdots & \vdots & \ddots & \vdots \\
0 & 0 & \cdots & \lambda_k I_{n_k}
\end{bmatrix}.
\end{equation*} 
Here, $\lambda_1, \dots, \lambda_k $ are the distinct eigenvalues of  the matrix $Y^{-1/2} V Y^{-1/2}$, which are ordered such that $\lambda_1 < \dots < \lambda_k$; $n_i$ is the multiplicity of $\lambda_i$; $I_{n_i} \in \mathbb{R}^{n_i \times n_i}$ is the identity matrix; and $U \in \mathbb{R}^{n \times n}$ is an orthogonal matrix. Let
$$
U^{T} Y^{-1/2} X Y^{-1/2} U = L L^{T}, 
$$
be the Cholesky decomposition. Then, the Busemann functions $B_{Y, V}$ evaluated at $X$ is given by
\begin{equation}\label{theo:bus.I.eq}
B_{Y, V}(X) = -2\left(  n_1\lambda_1^{2}+\cdots+n_k\lambda_k^{2}\right)^{-1/2} \sum_{i=1}^{k} \sum_{j=\alpha_{i-1} + 1}^{\alpha_i} {\lambda_i \ln\left( L_{jj} \right)}, 
\end{equation}
where $L_{jj}>0$ denotes the $j$-th diagonal entry of $L$, $\alpha_0 = 0$, and $\alpha_i = \sum_{j=1}^{i} n_j$ for $i = 1, \ldots, k$. Moreover, the Riemannian gradient of $B_{Y, V}$ at $X$ is given by
\begin{equation} \label{eq:gradbf}
{\rm grad} B_{Y, V}(X) = -\left(  n_1\lambda_1^{2}+\cdots+n_k\lambda_k^{2}\right)^{-1/2} Y^{1/2} U L D L^{T} U^{T} Y^{1/2}.
\end{equation}
\end{example}
For more explicit examples of Busemann functions, see, for instance \cite{Bento2023, BridsonHaefliger1999, hirai2023}. As noted in Remark~\ref{re:bfn1}, the examples in these references need to be adapted to fit our definition.

\section{Characterization for the subdifferential via Busemann functions} \label{sec:NewDefSub}

In this section, we provide a Busemann-function characterization of the classical subdifferential on Hadamard manifolds. More precisely, we show that the usual subgradient inequality can be equivalently expressed in terms of support functions built from Busemann functions. This viewpoint yields an intrinsic geometric representation of subgradients and equips the resulting support models with a more suitable structure for subsequent developments. In particular, the Busemann-based support function enjoys concavity  properties under our convention, a feature that will be crucial in the algorithmic analysis carried out in the next sections. 

It is well known that when the Hadamard manifold $\mathbb{M}$ has identically zero sectional curvature, the function appearing on the right-hand side of \eqref{eq:Sugd} in Definition~\ref{def:Sugd}, namely,
\begin{equation} \label{eq:Sugdlhs}
{\mathbb M} \ni p\mapsto f({{q}})+\left\langle s , \log_{ {q} } p\right\rangle,
\end{equation}
is affine in the sense that its Riemannian Hessian vanishes identically.  This affine support model is central to many developments in Euclidean nonsmooth optimization. However, when the curvature of ${\mathbb M}$ is nonzero, the function \eqref{eq:Sugdlhs} is no longer affine and, in general, it is neither geodesically convex nor geodesically concave; see \cite{Kristaly2016}.  This motivates the search for an alternative support representation that is more suitable for optimization on Hadamard manifolds. On the other hand, by Example~\ref{ex:defBFNDefK0}, when the sectional curvature of $\mathbb{M}$ is identically zero, i.e., $K \equiv 0$, we have the following equality
$$
\left\langle s , \log_{ {q} } p\right\rangle =  - \|s\| B_{q,s}(p).
$$
Although in general $\langle s, \log_q p \rangle \neq - \|s\|\,B_{q,s}(p)$, the flat-case identity suggests replacing the model \eqref{eq:Sugdlhs} by the Busemann-based support function
\begin{equation} \label{eq:Sugbf}
p \,\mapsto\, f(q) - \|s\| B_{q,s}(p), \qquad p \in \mathbb{M},
\end{equation}
with the aim of obtaining an intrinsic characterization of the classical subdifferential. In this context, Lemma~\ref{le:CharactBusFunc} shows that taking \eqref{eq:Sugbf} as a support function yields the concavity property required in our subsequent analysis. The next theorem establishes this characterization for $\sigma$-strongly convex functions; in particular, taking $\sigma=0$ recovers the convex case.

\begin{theorem}\label{th:ebesub}
Let $\mathbb{M}$ be a Hadamard manifold and let $f\colon \mathbb{M} \rightarrow \overline{\mathbb{R}}$ be a $\sigma$-strongly convex function with $\sigma\ge 0$. Then, for every $q \in \operatorname{dom}\, f$, the subdifferential of $f$ at $q$ admits the characterization
$$
 \partial f(q)=\bigl\{\, s\in T_{{q}}{\mathbb M}: ~ f({p}) \geq f({{q}})-\|s\| B_{{q},s}({p})+\tfrac{\sigma}{2}d^2(p,q), ~\forall p\in {\mathbb M}\bigr\}.
$$
\end{theorem}

\begin{proof}
We begin by proving 
$
\partial f(q) \subseteq \bigl\{\, s\in T_{{q}}{\mathbb M}:~ f({p}) \geq f({{q}})-\|s\| B_{{q},s}({p})+\frac{\sigma}{2}d^2(p,q), ~\forall p\in {\mathbb M}\bigr\}.
$
Let $s \in \partial f(q)$. Since $f$ is $\sigma$-strongly convex, by Theorem~\ref{thm:f-convex-subdiff} we have $f(p) \ge f(q) + \langle s,\log_q p\rangle + \tfrac{\sigma}{2}d^2(p,q)$, for all $p \in \mathbb{M}$.  On the other hand, Lemma~\ref{le:busfunc} guarantees that
$
-\|s\|\, B_{q,s}(p) \leq \langle s, \log_q p \rangle,
$
 for all $p \in \mathbb{M}$. Combining the last two inequalities gives $f(p) \ge f(q) - \|s\|\,B_{q,s}(p) + \tfrac{\sigma}{2}d^2(p,q)$,  for all $p\in\mathbb{M}$,  which shows that
 \begin{equation} \label{eq: fistinclusion}
  s\in \bigl\{\, s\in T_{{q}}{\mathbb M}:~ f({p}) \geq f({{q}})-\|s\| B_{{q},s}({p})+\tfrac{\sigma}{2}d^2(p,q), ~\forall p\in {\mathbb M}\bigr\}.
 \end{equation} 
 and the first inclusion follows. We now prove the reverse inclusion.  For that, take 
$$
s \in \bigl\{\, s\in T_{{q}}{\mathbb M}: ~ f({p}) \geq f({{q}})-\|s\| B_{{q},s}({p})+\tfrac{\sigma}{2}d^2(p,q), ~\forall p\in {\mathbb M}\bigr\}.
$$
If $s=0$, then $f(p)\geq f(q)+\frac{\sigma}{2}d^2(p,q)\geq f(q)$ for all $p\in\mathbb{M}$, which implies that  $0\in \partial f(q)$. Assume now that $s\neq 0$ and let $v\in T_q\mathbb{M}$. If $v\notin{\rm dom}\, f'(q,\cdot)$, then $f'(q;v)=+\infty$ and thus we conclude that 
$\langle s,v\rangle \le f'(q;v)$. Otherwise, take $v\in{\rm dom}\, f'(q,\cdot)$. Then there exists ${\hat t}>0$ such that
$\exp_q(tv)\in{\rm dom}\,f$ for all $t\in[0,{\hat t})$, and
$$
f(\exp_q(tv)) \ge f(q)-\|s\|\,B_{q,s}(\exp_q(tv))+\tfrac{\sigma}{2}d^2(\exp_q(tv),q), \quad \forall t\in[0,{\hat t}).
$$
Since $d(\exp_q(tv),q)=t\|v\|$ and $B_{q,s}(q)=0$, dividing both sides by $t$ yields
$$
\frac{f(\exp_q(tv))-f(q)}{t} \ge -\|s\|\,\frac{B_{q,s}(\exp_q(tv))-B_{q,s}(q)}{t}+\frac{\sigma}{2}\,t\|v\|^2,  \quad \forall t\in[0,{\hat t}).
$$
Taking the limit as $t$ goes to $0^+$ and using Lemma~\ref{le:CharactBusFunc}, namely ${\rm grad}\,B_{q,s}(q)=-s/\|s\|$, we obtain
$$
f'(q;v)\ge -\|s\|\langle {\rm grad}\,B_{q,s}(q),v\rangle=\langle s,v\rangle.
$$
Therefore, for all $v\in T_q\mathbb{M}$ we have $\langle s,v\rangle \le f'(q;v)$. Hence, by Proposition~\ref{pr:sfsd}, we conclude that
$s\in \partial f(q)$, which together with \eqref{eq: fistinclusion} completes the proof.
\end{proof}

As a direct consequence of Theorem~\ref{th:ebesub}, we obtain a variational characterization of subgradients in terms of global minimizers of a Busemann-based support model, with an additional quadratic term accounting for $\sigma$-strong convexity.
\begin{corollary}\label{cor:argmin-support}
Let $\mathbb{M}$ be a Hadamard manifold, let $f\colon \mathbb{M} \rightarrow \overline{\mathbb{R}}$ be a $\sigma$-strongly convex function with $\sigma\ge 0$, and let $q\in\operatorname{dom}\,f$. Then,  $s\in\partial f(q)$ if and only if
$$
q\in\operatorname*{argmin}_{p\in\mathbb{M}}\bigl\{\, f(p)+\|s\|B_{q,s}(p)-\tfrac{\sigma}{2}d^2(p,q)\,\bigr\}.
$$
As a consequence, if $s\in\partial f(q)$, then
$
f(q)=\min_{p\in\mathbb{M}}\bigl\{\,f(p)+\|s\|B_{q,s}(p)-\tfrac{\sigma}{2}d^2(p,q)\,\bigr\}.
$
\end{corollary}
\begin{proof}
Let $q\in\operatorname{dom}\, f$ and $s\in T_q\mathbb{M}$, and define the auxiliary function $\psi: \mathbb{M} \to \overline{\mathbb R}$ by
$$
\psi(p):=f(p)+\|s\|B_{q,s}(p)-\tfrac{\sigma}{2}d^2(p,q), \qquad p\in\mathbb{M}.
$$
Since $B_{q,s}(q)=0$ and $d(q,q)=0$, we have $\psi(q)=f(q)$. Assume first that $s\in\partial f(q)$. Thus, by Theorem~\ref{th:ebesub} we have
$$
f(p)\ge f(q)-\|s\|B_{q,s}(p)+\tfrac{\sigma}{2}d^2(p,q), \qquad \forall p\in\mathbb{M}.
$$
Rearranging and using the definition of $\psi$, we obtain $\psi(p)\ge \psi(q)$ for all $p\in\mathbb{M}$, which implies that
$q\in\operatorname*{argmin}_{p\in\mathbb{M}}\psi(p)$.

For the converse,  assume that $q\in\operatorname*{argmin}_{p\in\mathbb{M}}\psi(p)$. Then $\psi(p)\ge \psi(q)$ for all $p\in\mathbb{M}$, that is, $f(p)+\|s\|B_{q,s}(p)-\tfrac{\sigma}{2}d^2(p,q)\ge f(q)$, for all $p\in\mathbb{M}$, which is equivalent to
$$
f(p)\ge f(q)-\|s\|B_{q,s}(p)+\tfrac{\sigma}{2}d^2(p,q), \qquad \forall p\in\mathbb{M}.
$$
By Theorem~\ref{th:ebesub}, the latter implies $s \in \partial f(q)$ and the equivalence is proved.

For the last statement, if $s\in\partial f(q)$, then by the first part we have
$q\in\operatorname*{argmin}_{p\in\mathbb{M}}\psi(p)$, hence $\inf_{p\in\mathbb{M}}\psi(p)=\psi(q)$.
Since $\psi(q)=f(q)$, the conclusion follows.
\end{proof}

We next present  Corollary~\ref{cor:argmin-support} in perspective by comparing its variational characterization with the
Busemann-subgradient notion of \cite{Goodwin2024}, emphasizing the flat case where they coincide and the nonflat case where
they may differ.

\begin{remark}\label{re:argmin-support-busemannsubgrad}
In the context of Hadamard manifolds, Corollary~\ref{cor:argmin-support} is closely related in nature to the notion of
\emph{Busemann subgradient} in \cite[Definition~3.1]{Goodwin2024}, although the resulting supporting objects differ in general.
In the convex case $\sigma=0$, Corollary~\ref{cor:argmin-support} states that $s\in\partial f(q)$ if and only if
\begin{equation}\label{eq:char1}
q\in\arg \min_{p\in\mathbb{M}}\bigl\{\,f(p)+\|s\|\,B_{q,s}(p)\,\bigr\}.
\end{equation}
On the other hand, in the Hadamard manifold setting one may express \cite[Definition~3.1]{Goodwin2024} in our framework as follows:
a Busemann subgradient at a point  $x\in{\rm dom}\,f$ can be represented by a vector $\xi\in T_x\mathbb{M}$ such that
\begin{equation}\label{eq:char2}
x\in\arg\min_{y\in\mathbb{M}}\bigl\{\,f(y)-\|\xi\|\,B_{x,-\xi}(y)\,\bigr\}, 
\end{equation}
see also  \cite[Definition~1]{Criscitiello2025}. In the \emph{flat case} (sectional curvature identically zero), Example~\ref{ex:defBFNDefK0} yields
$$
\|s\|\,B_{q,s}(p)=-\langle s,\log_q p\rangle,
\qquad
\|\xi\|\,B_{x,-\xi}(y)=\langle \xi,\log_x y\rangle,
$$
and $B_{x,-\xi}=-B_{x,\xi}$. Hence \eqref{eq:char2} is equivalent to
$x\in\arg\min_{y\in \mathbb{M}}\{\,f(y)+\|\xi\|\,B_{x,\xi}(y)\,\}$, which coincides with \eqref{eq:char1} after the identification $q=x$ and
$s=\xi$. In particular, both notions recover the classical Euclidean supporting-hyperplane condition and the corresponding supports
coincide. In contrast, on \emph{nonflat} Hadamard manifolds \eqref{eq:char1} and \eqref{eq:char2} are not equivalent in general.
A key point is that \cite{Goodwin2024} requires a \emph{global} support parameterized by an ideal direction and a speed, and such a
support may fail to exist even for geodesically convex functions; see \cite[Example~3.2]{Goodwin2024}.
Consequently, if one defines the $b$-subdifferential at $x$ by
$
\partial^b f(x):=\bigl\{\,\xi\in T_x\mathbb{M}:~ \eqref{eq:char2}\ \text{holds}\,\bigr\},
$
or equivalently by 
$
\partial^b f(x):=\bigl\{\,\xi\in T_x\mathbb{M}: ~f(y)\ge f(x)+\|s\|B_{x,-\xi}(y), \forall y\in \mathbb{M} \bigr\},
$
then $\partial^b f(x)\subset \partial f(x)$, and this inclusion can be strict beyond the flat case.
\end{remark}

As a further consequence of Theorem~\ref{th:ebesub}, we recover the following classical bound linking Lipschitz continuity and the norm of subgradients.
\begin{corollary}\label{cor:subgrad-bound-lipschitz}
Let $f:\mathbb{M}\to\mathbb{R}$ be convex and $L$-Lipschitz on $\mathbb{M}$, i.e.,
$|f(x)-f(y)|\le L\,d(x,y)$ for all $x,y\in\mathbb{M}$. Then, for every $q\in\operatorname{dom}\,f$ and every $s\in\partial f(q)$, there holds $\|s\|\le L.$
\end{corollary}
\begin{proof}
Fix $q\in\operatorname{dom}\,f$ and $s\in\partial f(q)$. If $s=0$, the conclusion is trivial. Assume $s\neq 0$, and for $\tau>0$ set $p_\tau:=\exp_q(\tau s)$.
By Theorem~\ref{th:ebesub} and Lemma~\ref{eq:pbfu}, we have
$$
f(p_\tau)\ge f(q)-\|s\|B_{q,s}(p_\tau)=f(q)+\tau\|s\|^2.
$$
On the other hand, considering that $f$ is  a  $L$-Lipschitz function  on $\mathbb{M}$ we conclude  that 
$$
f(p_\tau)\le f(q)+L\,d(q,p_\tau)=f(q)+L\tau\|s\|.
$$
Combining the two inequalities and dividing by $\tau>0$ yields $\|s\|^2\le L\|s\|$, hence $\|s\|\le L$.
\end{proof}

As a further application of Theorem~\ref{th:ebesub}, we obtain an intrinsic characterization of $\sigma$-strong convexity in terms of Busemann-based support inequalities.
\begin{proposition}\label{pr:ccfc}
Let  $f\colon \mathbb M \rightarrow \overline{\mathbb{R}}$ be a function. Then $f$ is  $\sigma$-strongly convex with $\sigma\ge 0$ if and only if ${\rm dom}\, f$ is convex and, for every $q\in{\rm dom}\, f$, there exists ${v}\in T_q\mathbb M$ such that
$$
f(p)\ge f(q)-\|v\|\,B_{q,{v}}(p)+\tfrac{\sigma}{2}d^2(p,q), \qquad \forall p\in\mathbb M.
$$
In this case, $\partial f(q)\neq\emptyset$ for all $q\in{\rm dom}\, f$, and the inequality holds for every ${v}\in\partial f(q)$.
\end{proposition}
\begin{proof}
First assume that $f$ is $\sigma$-strongly convex. Then,  by Theorem~\ref{thm:f-convex-subdiff}, ${\rm dom}\,f$ is convex and we have
$\partial f(q)\neq\emptyset$ for all $q\in{\rm dom}\,f$. Fix $q\in{\rm dom}\,f$ and take $v\in\partial f(q)$. Applying Theorem~\ref{th:ebesub}
 we obtain
$$
f(p)\ge f(q)-\|{v}\|\,B_{q,{v}}(p)+\tfrac{\sigma}{2}d^2(p,q), \qquad \forall\,p\in\mathbb{M},
$$
which in particular shows the existence of such a vector at each $q$.

Conversely, assume that ${\rm dom}\,f$ is convex and that for every $q\in{\rm dom}\,f$ there exists ${v}\in T_q\mathbb{M}$ such that
$$
f(p)\ge f(q)-\|{v}\|\,B_{q,{v}}(p)+\tfrac{\sigma}{2}d^2(p,q), \qquad \forall\,p\in\mathbb{M}.
$$
By Theorem~\ref{th:ebesub}, it follows that ${v}\in\partial f(q)$ for every $q\in{\rm dom}\,f$.
In particular, $\partial f(q)\neq\emptyset$ for all $q\in{\rm dom}\,f$. Hence, by Definition~\ref{def:Sugd}, for each $q\in{\rm dom}\,f$
there exists $u\in T_q\mathbb{M}$ (namely, $u:={v}$) such that
$$
f(p)\ge f(q)+\big\langle u,\log_q p\big\rangle+\tfrac{\sigma}{2}d^2(p,q), \qquad \forall\,p\in\mathbb{M}.
$$
Applying Theorem~\ref{thm:f-convex-subdiff}, we conclude that $f$ is $\sigma$-strongly convex. Finally, once $\sigma$-strong convexity is established,
Theorem~\ref{th:ebesub} ensures that the Busemann support inequality holds for every ${v}\in\partial f(q)$.
\end{proof}

We conclude this section by comparing our Busemann support characterization of geodesic convexity with horospherical convexity \cite{Criscitiello2025}, noting that they coincide in the flat case but may differ in the nonflat case.

\begin{remark}\label{re:ccfc-hconvex}
Proposition~\ref{pr:ccfc} is conceptually different from the horospherical support inequality in \cite[Definition~1]{Criscitiello2025}, although it plays an analogous supporting-role; it provides a  geodesic convexity characterization via Busemann-based supports. Indeed, when $\sigma=0$, Proposition~\ref{pr:ccfc} asserts that for every $q\in{\rm dom}\,f$ there exists $v\in T_q\mathbb{M}$ such that $f(p)-f(q)\geq -\|v\|\,B_{q,v}(p)$, for all $p\in\mathbb{M}$.  For comparison,  the defining inequality of $h$-convexity in \cite[Eq.~(6)]{Criscitiello2025} can be written as $f(p)-f(q)\ge\|v\|\,B_{q,-v}(p)$. Moreover, by Example~\ref{ex:defBFNDefK0}, in the \emph{flat case} (sectional curvature identically zero) one has
$
B_{q,v}(p)=\langle v,\log_q p\rangle,
$
and, in this setting, $B_{q,-v}=-B_{q,v}$, so the two support inequalities coincide and both reduce to the classical affine support characterization of convexity in $\mathbb{R}^n$. In contrast, on \emph{nonflat} Hadamard manifolds the antisymmetry $B_{q,-v}=-B_{q,v}$ generally fails, so horospherical supports built from $B_{q,-v}$ and the Busemann supports in Proposition~\ref{pr:ccfc}, which are expressed in terms of $-B_{q,v}$, need not coincide; see \cite[Sec.~3.5]{Criscitiello2025} for a conceptual discussion of the genuinely global nature of $h$-convexity. Finally, Proposition~\ref{pr:ccfc} guarantees that, in the  convex case, the support inequality holds for \emph{every} Riemannian subgradient $v\in\partial f(q)$, whereas the horospherical subdifferential $\partial^h f(q)$ in \cite{Criscitiello2025} consists only of those directions producing a global horospherical support. In fact, we can show that $\partial^h f(q)\subset \partial f(q)$, and this inclusion can be strict beyond the flat case.
\end{remark}

\section{The Busemann DC Algorithm for DC optimization} \label{sec:OptProhm}
In this section, we develop a Busemann-based DC algorithm for DC optimization on Hadamard manifolds.
The main  idea is to replace the standard linearization step in classical DC schemes by a Busemann-type support term,
thereby producing geodesically convex subproblems that better reflect the ambient geometry.
This yields a geometric DCA framework and a practical algorithmic tool for DC programs on Hadamard manifolds. We consider the {\it difference-of-convex (DC)} optimization problem on a Hadamard manifold $\mathbb{M}$, defined by
\begin{equation}\label{Pr:DCproblem}
{\rm argmin}_{p\in {\mathbb M}} \,\phi(p), \qquad \text{where}\quad  \phi(p)\coloneqq g(p)-h(p),
\end{equation}
where $g,h\colon{\mathbb M}\to\overline{\mathbb{R}}$ are proper, lower semicontinuous, and geodesically convex. Despite $\phi$ being the difference of two convex functions, $\phi$ generally does not exhibit convexity. However, we can show that $\phi$ is locally Lipschitz on the interior of its domain.  Consequently, $\phi$ possesses a subdifferential in the Clarke sense, denoted by $\partial^o{\phi}$, as detailed in \cite{Bento2010, Bento2015}. Moreover, it can be demonstrated that $\partial^o\phi(p)\subset \partial g(p)-\partial h(p)$. Therefore, a requisite condition for a point $p^{*}\in {\mathbb M}$ to qualify as a local minimum of $\phi = g-h$ is that $0\in  \partial^o  \phi(p^{*})   \subset \partial g(p^{*})-\partial h(p^{*})$. 
Thus,  leading us to define a \textit{critical point} of problem~\eqref{Pr:DCproblem} as follows: 

\begin{definition}%
\label{def:critpoint}
A $p^{*}\in {\mathbb M}$ is called  a critical point for the problem \eqref{Pr:DCproblem} if $\partial g(p^{*}) \cap \partial h(p^{*}) \neq\varnothing$.
\end{definition}

\noindent
For further discussion on the definition of a  critical point for  problem \eqref{Pr:DCproblem}, see \cite{DeOliveira2020} for example. In the following two sections, we employ the Busemann functions to introduce and analyze two variants of the EDCA for solving problem~\eqref{Pr:DCproblem}. In addition to properness, lower semicontinuity, and convexity, we work under the following assumptions:
   \begin{enumerate}[label={A\arabic*)}, ref={(A\arabic*)}]
     \item%
    \label{it:A1} {\it ${\mathbb M}$ is  a  Hadamard   manifold}.
    \item%
    \label{it:A2} ${\phi}_{\inf}\coloneqq \displaystyle\inf_{x\in{\mathbb M}}
         \phi (x) > -\infty$;
    \item%
    \label{it:A3}  ${\rm dom }\, g$ and  ${\rm dom}\,h$ are convex  and ${\rm dom }\, g \subseteq \mbox{int} \, {\rm dom} \, h.$
    \end{enumerate}
    Firstly, we address the assumptions outlined above.  Assumption~\ref{it:A1} is fundamental for the analysis of the algorithms proposed in the following sections, since it is repeatedly used through Theorem~\ref{th:ebesub}. Under assumption \ref{it:A2}, the domain of $\phi$ equals the domain of $g$, that is, ${\rm dom}\, \phi = {\rm dom }\, g \subseteq {\rm dom}\, h$. This inclusion is derived from the fact that if ${\rm dom } \,g \not\subseteq {\rm dom} \, h$, then there exists a point $p \in {\rm dom } \, g$ such that $p \notin {\rm dom}\, h$. According to \eqref{eq:conventions}, this situation implies $\phi(p) = g(p) - h(p) = g(p) - (+\infty) = -\infty$, which is in contradiction with assumption \ref{it:A2}. Consequently, we have ${\rm dom } \, g \subseteq {\rm dom}\, h$, and thus ${\rm dom } \,g \subseteq {\rm dom}\, \phi$. 

 Conversely, assume for contradiction that ${\rm dom}\, \phi \not\subseteq {\rm dom } \, g$. Then, there must be some $p \in {\rm dom}\, \phi$ for which $g(p) = +\infty$. Utilizing \eqref{eq:conventions} again, we find $\phi(p) = g(p) - h(p) = (+\infty) - h(p) = +\infty$. However, this contradicts the definition of the domain of $\phi$, since $p$ would not be included in ${\rm dom} \, \phi$ if $\phi(p)$ were infinity. Thus, ${\rm dom}\, \phi \subseteq {\rm dom }\, g$. Therefore, ${\rm dom}\, \phi = {\rm dom } \, g$, and since ${\rm dom } \, g \subseteq {\rm dom}\, h$ under assumption \ref{it:A2}, it follows that assumption \ref{it:A3} is only marginally more restrictive than \ref{it:A2}. Additionally, if ${\rm dom} \, h = \mathbb{M}$, then assumption \ref{it:A3} is naturally satisfied.

\subsection{The Busemann DC algorithm  } \label{sec:DCA}
In this subsection, we begin by revisiting the classical difference-of-convex algorithm on Hadamard manifolds, introduced and analyzed in \cite{Bergmann2024} as a manifold counterpart of the EDCA. Since the function involved in the subproblem  of the classic DCA on Hadamard manifolds is not  convex  in general, it presents challenges in seeking solutions. Here, by using Busemann functions, we propose a new version of this method to overcome this limitation. Specifically, the function in the subproblem is now convex, enabling more effective optimization within the Riemannian context. To propose and analyze the method, we assume that the functions in problem \eqref{Pr:DCproblem} satisfy the following hypothesis:
\begin{enumerate}[label={(H\arabic*)}, ref={(H\arabic*)}]
    \item%
    \label{it:H1} $g\colon{\mathbb M} \to \overline{\mathbb{R}}$  and $h\colon{\mathbb M} \to \overline{\mathbb{R}}$ are $\sigma$-strongly convex and lsc, where $\sigma>0$;
        \end{enumerate}
We begin by demonstrating that \ref{it:H1} is not restrictive. In fact,  let $q \in {\mathbb M}$ and $\sigma >0$. Consider the function ${\mathbb M}\ni p \mapsto ({\sigma}/{2})d^{2}(q, p)$, which is $\sigma$-strongly convex, as demonstrated in~\cite[Corollary 3.1]{NetoFerreiraPerez2002}. If $\tilde g \colon{\mathbb M}\to \overline{\mathbb{R}}$ and $\tilde h \colon{\mathbb M}\to \overline{\mathbb{R}}$ are convex, then by choosing $q\in {\mathbb M}$ and defining $g(p)={\tilde g}(p) + ({\sigma}/{2})d^{2}(q, p)$ and $h(p)={\tilde h}(p)+ ({\sigma}/{2})d^{2}(q, p)$, we obtain two $\sigma$-strongly convex functions $g$ and $h$ on ${\mathbb M}$. Furthermore, it holds that $\phi ={\tilde g}-{\tilde h}=g-h$. However, it is imperative to caution users of the Riemannian difference of convex algorithm proposed below that the selection of the parameter $\sigma>0$ significantly influences the convergence rate of the methods employed to solve the subproblem, as well as the overall method. 

Next, we revisit a Riemannian version of the EDCA, which, using the same terminology, we refer to it as the {\it classic Riemannian Difference of Convex Algorithm} (CR-DCA). This algorithm is used to solve the DC problem \eqref{Pr:DCproblem}, which is formally stated as follows:

\begin{algorithm}[H]
	\begin{footnotesize}
	\begin{description}
		 \item[Step 0.] Set $p_0 \in {\rm dom } \, g$ and  set $k \gets 0$;
		 \item[Step 1.]    Take $s_{k}\in\partial h(p_{k})$, and define the next iterate $p_{k+1}$  as follows 
		  \begin{equation} \label{eq:DCAS}
                p_{k+1}={\rm arg min}_{p\in {\mathbb M}}
                    \left(g(p)-\left\langle s_{k}, \log_{p_{k}}p\right\rangle\right).
            \end{equation}
		 \item[Step 2.]  If $p_{k+1}=p_{k}$, then {\bf stop} and return the point $p_{k}$. Otherwise, set $k\gets k+1$ and go to \textbf{Step~1}.
\end{description}
	\caption{Classic Riemannian difference of convex algorithm (CR-DCA)}
	\label{Alg:DCAC}
	\end{footnotesize}
\end{algorithm}  

It can be shown that in  the Euclidean setting, the inverse of the exponential mapping is   given by ${\mathbb M}\ni p \mapsto \log_{p_{k}}p=p-p_{k}$. Therefore, Algorithm~\ref{Alg:DCAC},  represents a version  in Hadamard manifold of the EDCA. The motivation behind introducing the EDCA lies in the fact that the function ${\mathbb M}\ni p \mapsto g(p)-\left\langle s_{k}, p-p_{k}\right\rangle$ in the subproblem \eqref{eq:DCAS}   in the Euclidean setting  is convex. By doing so, one replaces the solution of the non-convex problem \eqref{Pr:DCproblem} with the solution of a sequence of convex subproblems. However, as discussed in Section~\ref{sec:NewDefSub}, the function in the subproblem \eqref{eq:DCAS} is not  convex  on general Hadamard manifolds, posing challenges in its solution.  To address this limitation,  we can redefine the CR-DCA by employing a Busemann functions, wherein the function in the subproblem counterpart of \eqref{eq:DCAS} becomes convex. The formulation of the {\it Busemann Difference of Convex Algorithm} (B-DCA) to solve the DC problem \eqref{Pr:DCproblem} is outlined below:

\begin{algorithm}[H]
	\begin{footnotesize}
	\begin{description}
		 \item[Step 0.] Set $p_0\in {\rm dom } \,g$ and  set $k \gets 0$;
		 \item[Step 1.]    Take $s_{k}\in\partial h(p_{k})$, and define the next iterate $p_{k+1}$  as follows 
		  \begin{equation} \label{eq:B-DCAS}
                p_{k+1}:={\rm arg min}_{p\in {\mathbb M}}
                    \left(g(p)+\|s_k\|B_{p_k, s_k}(p)\right).
            \end{equation}
		 \item[Step 2.]  If $p_{k+1}=p_{k}$, then {\bf stop} and return the point $p_{k}$. Otherwise, set $k\gets k+1$ and go to \textbf{Step~1}.
\end{description}
	\caption{Busemann difference of convex algorithm (B-DCA) }
	\label{Alg:B-DCAC}
	\end{footnotesize}
\end{algorithm}  
It is noteworthy that in the case where the curvature of the Hadamard manifold is $K=0$, Example~\ref{ex:defBFNDefK0} establishes the equivalence between Algorithm~\ref{Alg:DCAC} and Algorithm~\ref{Alg:B-DCAC}. In particular,    Example~\ref{ex:defBFNDefK0} serves to illustrate that B-DCA aligns with EDCA.    Furthermore, comparing the objective in subproblem~\eqref{eq:DCAS} with that in subproblem~\eqref{eq:B-DCAS}, we introduce the function ${\phi}_{k}\colon{\mathbb M}\to \overline{\mathbb{R}}$ defined by
\begin{equation}\label{eq:afwd}
{\phi}_{k}(p)\coloneqq g(p)+\|s_k\|\,B_{p_k, s_k}(p).
\end{equation}
Then, in view of Assumption~\ref{it:H1}, the function ${\phi}_{k}$ is  $\sigma$-strongly convex.  Before proceeding with the analysis of Algorithm~\ref{Alg:B-DCAC}, it is imperative to acknowledge that the point $p_{k+1}$, as a solution of \eqref{eq:B-DCAS}, satisfies the following inequality:
\begin{equation}\label{eq:sol}
g(p) + \|s_k\| B_{p_k, s_k}(p) \geq g(p_{k+1}) + \|s_k\| B_{p_k, s_k}(p_{k+1})
\qquad \forall p \in {\mathbb M}.
\end{equation}
To further advance the analysis of Algorithm~\ref{Alg:B-DCAC}, we first establish that the algorithm is well defined. This fundamental property is addressed in the following proposition.
\begin{proposition} \label{pr:WellDefB-DCA}
    Algorithm~\ref{Alg:B-DCAC} is well defined, i.e., $p_{k}\in \mbox{dom }\,g$ for all $k=0,1,\ldots$, and the subproblem in  \eqref{eq:B-DCAS} has a unique solution. Moreover, if $p_{k+1} =p_{k}$, then $p_{k}$ is a critical point of $\phi$.
\end{proposition}
\begin{proof}
    Assume that $p_{k}\in {\rm dom } \,g$.  Using assumption~\ref{it:A3} we conclude that $p_{k}\in \mbox{int} \, {\rm dom} \, h$. Thus, by \cite[Theorem~3.3]{Ferreira2002},   we have $\partial h(p_{k})\neq \varnothing$. Let $s_{k}\in \partial h(p_{k})$. Now we are going to show that the subproblem in \eqref{eq:B-DCAS} has a unique solution. Since, by \ref{it:H1}, $g\colon {\mathbb M} \to\mathbb{R}$ is $\sigma$-strongly convex, it follows from Theorem~\ref{thm:f-convex-subdiff} that 
    \begin{equation*}
     g(p)\geq g({p_k}) + \langle {v},\log_{p_k}p\rangle + \frac{\sigma}{2}d^2({p_k},p),\qquad \forall p \in {\mathbb M}, ~\forall {v}\in \partial g({p_k}). 
    \end{equation*} 
  Thus, by considering ${\phi}_{k}(p)= g(p)+\|s_k\|B_{p_k, s_k}(p)$ and employing the last inequality, we deduce 
  \begin{equation} \label{eq:dcgkf} 
  \frac{{\phi}_{k}(p)}{d({p_k},p)}\geq \frac{g({p_k})}{d({p_k},p)} + \Big\langle {v},\frac{\log_{p_k}p}{d({p_k},p)}\Big\rangle+ \frac{\sigma}{2}d({p_k},p) + \|s_k\|\frac{B_{p_k, s_k}(p)}{{d({p_k},p)}}, \qquad \forall p \in {\mathbb M}, \forall {v}\in \partial g({p_k}). 
  \end{equation}
   Observe that \eqref{eq:defFIWDef2b} yields $|B_{p_k, s_k}(p)|\leq d(p_k, p)$, for all $p\in {\mathbb M}$, and $d({p_k},p)=\|\log_{p_k}p\|$. Thus, taking the limit in \eqref{eq:dcgkf} we obtain that $\lim_{d({p_k},p)\to+\infty} {{\phi}_{k}(p)}/{{d({p_k},p)}} = +\infty$. In particular, we conclude that $ \lim_{d({p_k},p)\to+\infty } {{\phi}_{k}(p)} = +\infty$. Hence, using Proposition~\ref{prop:CoeFunc}  we conclude that  ${\phi}_{k}$ has  a  global minimizer.  Therefore, the subproblem in  \eqref{eq:B-DCAS} has a global  solution and,  due to  ${\phi}_{k}$  be  $\sigma$-strongly convex,   the solution is unique.  Consequently, there exists a unique  $p_{k+1}\in {\rm dom }\, g={\rm dom} \, \phi$  satisfying \eqref{eq:B-DCAS}, which implies that  Algorithm~\ref{Alg:B-DCAC} is well defined. To prove the last statement, we assume that $p_{k+1} =p_{k}$.    Thus, \eqref{eq:sol} implies that $g(p)\geq  g(p_{k}) -\|s_k\| B_{p_k, s_k}(p)$, for all $p\in {\mathbb M}$, which by using Theorem~\ref{th:ebesub}  shows that $s_{k}\in \partial g(p_{k})$. Hence, taking into account that $s_{k}\in\partial h(p_{k})$, we conclude that $s_{k}\in \partial  g(p_{k}) \cap \partial  h(p_{k}) \neq~\varnothing$. Therefore, it follows from Definition~\ref{def:critpoint} that $p_{k}$ is a critical point of $\phi$ in problem~\eqref{Pr:DCproblem}.
\end{proof}

Drawing on Proposition~\ref{pr:WellDefB-DCA}, we establish the well-defined nature of Algorithm~\ref{Alg:B-DCAC}, resulting in the generation of a sequence $(p_{k})_{k\in \mathbb{N}}$. Given that the termination condition implies the critical point of $\phi$ is attained in the final iteration, {\it we anticipate this sequence will  be infinite, an assumption  adopted henceforth}.

\subsection{Convergence analysis}
In this section, we conduct a detailed analysis of the B-DCA, designated as Algorithm~\ref{Alg:B-DCAC} under the assumptions \ref{it:A1}, \ref{it:A2}, \ref{it:A3} and \ref{it:H1}. The theoretical results obtained herein correspond to those obtained for the EDCA and CR-DCA versions. Nonetheless, it is noteworthy that in the B-DCA variant, the utilization of the Busemann functions renders the subproblem convex, a notable departure from the CR-DCA version. This  fundamental  advancement simplifies the process of solving subproblems, given that convex problems have significantly lower computational complexity than non-convex ones.  We begin by demonstrating a descent property of the algorithm and establishing that the distance  between   consecutive iterates converges  to zero as $k$ tends to infinity.

\begin{proposition}\label{pr:ffrdp}
Let $(p_k)_{k\in \mathbb{N}}$ be a sequence generated by Algorithm~\ref{Alg:B-DCAC}. Then, it satisfies the inequality
\begin{equation}\label{eq:dsc}
\phi(p_{k+1}) \leq \phi(p_{k}) - \frac{\sigma}{2}d^{2}(p_{k},p_{k+1}), \quad \forall k\in \mathbb{N}.
\end{equation}
As a consequence, the sequence $(\phi(p_{k}))_{k\in \mathbb{N}}$ is strictly decreasing and converges. Moreover, we have $\lim_{k\to +\infty} d(p_{k},p_{k+1}) = 0$.
\end{proposition}
\begin{proof}
Considering that   $d(p_k, p_k)=0$, from  \eqref{eq:defFIWDef2b}  we have   $B_{p_k, s_k}(p_{k})=0$. Thus,  using the inequality in~\eqref{eq:sol} with  $p=p_{k}$ we conclude that 
$
g(p_{k})  \geq g(p_{k+1}) + \|s_k\| B_{p_k, s_k}(p_{k+1}).
$
Moreover, taking into account that   $h$ is   a $\sigma$-strongly convex function and $s_{k}\in \partial h(p_{k})$, it follows from  Theorem~\ref{th:ebesub}  that 
$
h(p_{k+1}) \geq  h({p_{k}}) -\|s_k\| B_{p_k, s_k}(p_{k+1})+(\sigma/2)d^2(p_{k+1},p_{k}).
$
Thus, combining the  two  previous inequalities and using that  $\phi=g-h$, we derive~\eqref{eq:dsc}.  To verify the second statement, we first observe that \eqref{eq:dsc} entails the inequality 
\begin{equation} \label{eq:CmIneq1}
0\leq (\sigma/2)d^{2}(p_{k},p_{k+1})\leq \phi(p_{k})-\phi(p_{k+1}), \qquad \forall k\in \mathbb{N}.
\end{equation}
Since we are assuming that  $(p_{k})_{k\in {\mathbb N}}$ is  infinite,   it follows from Proposition~\ref{pr:WellDefB-DCA} that
$p_{k+1}\neq p_{k}$.   Hence,  considering that $\sigma>0$, we have  from
\eqref{eq:CmIneq1} that  $\phi(p_{k+1})<\phi( p_{k})$,  for all $k\in
\mathbb{N}$. Therefore,   $(\phi(p_{k}))_{k\in {\mathbb N}}$  is strictly
decreasing.  Furthermore,  due to  \ref{it:A2}  implying
that $(\phi(p_{k}))_{k\in {\mathbb N}}$ is bounded from below, we conclude that
it converges.  Since   $(\phi(p_{k}))_{k\in {\mathbb N}}$   converges, taking the limit in \eqref{eq:CmIneq1} we obtain that $\lim_{k\to +\infty}
d(p_{k},p_{k+1})=0$, which concludes the proof.
\end{proof}

The following theorem is the main finding regarding the convergence behavior of Algorithm~\ref{Alg:B-DCAC}, detailing the limiting behavior of the produced sequences and establishing their relationship with critical points of the objective function.
\begin{theorem} \label{th:cdca}
    Let $(p_{k})_{k\in \mathbb N}$ and $(s_{k})_{k\in \mathbb N}$ be generated
    by Algorithm~\ref{Alg:B-DCAC}. If ${\bar p}$ is a cluster point of $(p_{k})_{k\in \mathbb{N}}$,
    then ${\bar p}\in {\rm dom } \, g$ and there exists a cluster point ${\bar s}$ of
    $(s_{k})_{k\in \mathbb N}$ such that
    ${\bar s}\in \partial g({\bar p})\cap \partial h({\bar p})$.
    Consequently, every cluster point of $(p_{k})_{k\in \mathbb{N}}$,
    if any, is a critical point of $\phi$.
\end{theorem}
\begin{proof}
    Let ${\bar p}\in {\mathbb M}$ be a cluster point of the sequence $(p_{k})_{k\in \mathbb{N}}$. We may assume, without loss of generality, that  $\lim_{k \to +\infty}p_{k}={\bar p}$.  It follows from Proposition~\ref{pr:ffrdp}  that $(\phi(p_{k}))_{k\in \mathbb N}$ is strictly decreasing  and convergent. Moreover, due to $\phi(p_{0})\geq \phi(p_{k})=g(p_{k})-h(p_{k})$  we have  $\phi(p_{0})+h(p_{k})\geq g(p_{k})$. Thus, considering that  $g$ and $h$ are  lsc, we have    
    \begin{equation*}
      \phi(p_{0})+h ({\bar p})=  \liminf_{k\to +\infty}( \phi(p_{0})+h(p_{k})) \geq \liminf_{k\to +\infty} g(p_{k}) \geq  g({\bar p}), 
    \end{equation*}
    which  implies that $ \phi(p_{0})\geq  \phi({\bar p})$. Thus, due to $p_0\in {\rm dom } \, g= {\rm dom} \, \phi$ we conclude that  ${\bar p}\in {\rm dom} \, \phi={\rm dom } \,g$.  Hence, using  that  ${\bar p}\in{\rm dom }\, g$  and \ref{it:A3}, we  also have  ${\bar p}\in {\rm int} \, {\rm dom} \, h$.  By Proposition~\ref{pr:WellDefB-DCA} together with  \ref{it:A3}, we know that $(p_{k})_{k\in \mathbb N} \in {\rm int} \, \mbox{dom} \, h$, and {\bf Step~1} implies that   $s_{k}\in\partial  h(p_{k})$, for all $k\in \mathbb{N}$.  Therefore, if necessary, by considering a subsequence, we can assume, without loss of generality, by invoking Proposition~\ref{cont_subdif}, that  $\lim_{k\to+\infty} s_{k}={\bar s}\in \partial h({\bar p})$.   On the other hand, due to the point $p_{k+1}$ being a solution of~\eqref{eq:B-DCAS}, it satisfies \eqref{eq:sol}, which  considering  \eqref{eq:defFIWDef2b}  implies that 
\begin{equation}\label{eq:aplbc}
g(p) \geq g(p_{k+1}) -  \|s_k\| B_{p_k, s_k}(p) - \|s_k\| d(p_{k},  p_{k+1}) \qquad \forall p \in {\mathbb M}.
\end{equation}
Hence, by taking the inferior limit in~\eqref{eq:aplbc}, as $k$ goes to $+\infty$, and using the facts that $\lim_{k \to +\infty} p_k = \bar{p}$, $\lim_{k \to +\infty} s_k = \bar{s}$, $\lim_{k \to +\infty} d(p_k, p_{k+1}) = 0$, the function $g$ is lower semicontinuous, and invoking Lemma~\ref{le:cbf}, we conclude that
$$
g(p) \geq g(\bar{p}) - \|\bar{s}\|\, B_{\bar{p}, \bar{s}}(p), \qquad \forall p \in \mathbb{M}.
$$
Hence, it follows from  Theorem~\ref{th:ebesub}  that  $\bar{s} \in \partial g(\bar{p})$. Therefore, given that we already know ${\bar s} \in \partial  h({\bar p})$, it follows that ${\bar s} \in \partial  g({\bar p}) \cap \partial  h({\bar p})$. This confirms that ${\bar p}$ is a critical point of problem~\eqref{Pr:DCproblem}, thus completing the proof.
\end{proof}
In light of Proposition~\ref{pr:WellDefB-DCA}, we can see that the quantity  $d(p_{k},p_{k+1})$  can be interpreted as a measure of the criticality of point $p_k$. The following proposition provides a bound on the iteration complexity for this measure.
\begin{proposition}
    Let $(p_{k})_{k\in \mathbb N}$  be generated
    by Algorithm~\ref{Alg:B-DCAC}. 
    Then, for all $N\in \mathbb{N},$ there holds 
    $$
    \min_{k=0,1,\ldots,N}   d(p_{k},p_{k+1}) \leq \Big( \frac{ 2(\phi(p_{0})-{\phi}_{\inf}}{\sigma(N+1)} \Big)^{\frac{1}{2}}.
    $$
\end{proposition}
\begin{proof}
 From \eqref{eq:dsc}, we have $d^{2}(p_{k},p_{k+1})\leq (2/\sigma) \left( \phi(p_{k})-\phi(p_{k+1}) \right)$,
    for all $k\in \mathbb{N}$. Consequently,
    \begin{equation*}
            (N+1)\min _{k=0,1,\ldots,N}\bigl( d^{2}(p_{k},p_{k+1}) \bigr) \leq  \sum_{k=0}^{N}\frac{2}{\sigma} \bigl( \phi(p_{k})-\phi(p_{k+1}) \bigr) \leq \frac{2}{\sigma}\bigl(\phi(p_{0})-{\phi}_{\inf}\bigr),
    \end{equation*}
    for all $N\in \mathbb{N}$,  where   $-{\phi}_{\inf}<+\infty$ is determined by \ref{it:A2}. Hence, the desired inequality follows.
\end{proof}

\section{Numerical experiments} \label{sec:Numerics}
In this section we investigate the numerical performance of Algorithms~\ref{Alg:DCAC} (CR--DCA) and \ref{Alg:B-DCAC} (B-DCA) on a collection of DC optimization problems posed on the $\kappa$-hyperbolic space $\mathbb H_{\kappa}^{n}$ and on the manifold of symmetric positive definite matrices $\mathcal P(n)$.  In all experiments, the function $g$ is differentiable, and $h$ is differentiable; 
the solver uses closed-form expressions for the Riemannian gradients of the subproblem objectives \eqref{eq:DCAS} and \eqref{eq:B-DCAS}. The empirical comparisons should be read in light of the analytical guarantees established for B-DCA in this work, in particular the geodesic convexity of the inner subproblem.

Algorithms \ref{Alg:DCAC} and \ref{Alg:B-DCAC} were implemented in Matlab R2019a,  and numerical experiments carried out in a Intel Core i5 1.8Ghz, 8 GB Ram, running MAC OS X 10.13.6.   We have used Manopt Version~8~\cite{manopt}, and the subproblems \eqref{eq:DCAS} and \eqref{eq:B-DCAS} were solved by the \texttt{trustregions} solver with its default options; among the solvers available in Manopt, this choice consistently exhibited the most stable behavior for these subproblems. The objectives were scaled by a factor $\gamma = 1/(\| \text{grad} \phi(p_0) \| + 1)$  where $p_0 \in \mathbb{M}$ denotes the starting point. As stopping criteria we used $\| \text{grad} \phi(p_k) \| \leq \varepsilon$ or $d(p_{k+1},p_k) \leq \varepsilon$, with $\varepsilon = \gamma \times 10^{-4}$. Both algorithms are run with the same  starting point $p_{0}$, scaling $\gamma$, stopping thresholds, and Manopt \texttt{trustregions} options. 
For each instance we report the number of outer iterations $k$, the number of inner iterations \texttt{inn} for solving the subproblems, the ratio \texttt{inn}/$k$, the final function value \texttt{fval}, the final scaled norm of the Riemannian gradient \texttt{grad}, and, when available, the running time (\texttt{time}) in seconds. The pair (\texttt{fval}, \texttt{grad}) certifies solution quality; $(k,\texttt{inn},\texttt{inn}/k)$ measures solver workload; \texttt{time} summarizes the net computational cost. Differences in $k$ should be interpreted together with \texttt{inn}/$k$. As will be seen throughout this section, across all reported tests CR--DCA and B--DCA attain comparable solution quality, as certified by the terminal objective values \texttt{fval} and gradient norms \texttt{grad}. The workload measures $k$, \texttt{inn}, and \texttt{inn}/$k$ are of the same order for both methods within each problem family; when differences occur, they are mainly in the outer count $k$, while \texttt{inn}/$k$ remains similar. When total runtime \texttt{time} is reported, it aligns with these workload indicators and shows no systematic dominance. Overall, the tables indicate that B–DCA is consistently competitive, matching the accuracy of CR–DCA with comparable solver effort, which underscores its relevance within the proposed DC framework.

The codes and data for the numerical experiments are available at \url{http://mtm.ufsc.br/~douglas/downloads/BusemannDCA/}.

\subsection{Test problems in hyperbolic space}\label{sec:hyperbolic}
Here we consider a Rosenbrock-type function in the  $\kappa$-hyperbolic space form as the DC test problem. 
We recall that basic results on $\mathbb{H}_{\kappa}^n$ were presented in Section~\ref{sec:int.1}, and properties of the Busemann functions in this Hadamard manifold were discussed in Example~\ref{ex:BusHype}.
Before stating the test problems, we also show a useful result.

Consider the function the function $p \mapsto d_{\kappa}(p, q):=  (1/ \sqrt{\kappa})\operatorname{arcosh} (-\kappa \langle p,q\rangle )$. Thus, using \eqref{eq:grad}, we have 
\begin{equation} \label{eq:ddh}
{\rm grad}  d_{\kappa}^2(p,q)=-2\,\frac{\sqrt{\kappa}d_{\kappa}(p, q)}{\sqrt{\kappa^2\langle p,q\rangle^2-1}}(q+\kappa \langle p, q\rangle p), \qquad p\neq  q.
\end{equation}
\subsubsection{DC problem with a $\theta$-order hyperbolic Rosenbrock-type objective} \label{sec:hyper_rosen_theta}
The Rosenbrock function is a classical benchmark in nonlinear optimization, known for its nonconvex structure and narrow, curved valley that make it particularly challenging for iterative algorithms. Motivated by its role in the Euclidean setting, we introduce a family of intrinsic analogues defined on hyperbolic space, which preserve the essential geometric and analytical features of the original function. These hyperbolic Rosenbrock-type functions are formulated as DC  functions, making them well suited for evaluating the performance of DC optimization methods in negatively curved geometries. We provide  closed-form DC decompositions, establish the geodesic convexity of the components, and explicitly characterize the global minimizers. These constructions provide a principled extension of classical benchmarks to non-Euclidean settings, with applications to evaluating optimization algorithms on manifolds. With this in mind, we now present a detailed discussion. For clarity, our analysis focuses on the hyperbolic space ${\mathbb H}^n_1 =: \mathbb{H}^n$.

Let $\theta \ge 1$ be a fixed real exponent. Extending the classical Rosenbrock function  we introduce the $\theta$‑order $(a,b,\theta)$-family of hyperbolic Rosenbrock-type function defined on $\mathbb{H}^n$ by
\begin{equation}\label{eq:defrf}
f(p) := \bigl(a - d(p, \bar{p})^{\theta} \bigr)^2 + b \bigl(d(p, \bar{q})^{\theta} - d(p, \bar{p})^{2\theta} \bigr)^2, \qquad p \in \mathbb{H}^n,
\end{equation}
where $a, b > 0$ with $b \gg a \ge 1$, and $\bar{p}, \bar{q} \in \mathbb{H}^n$ are two fixed points satisfying the following conditions 
\begin{equation}\label{eq:radii-theta}
a^{\frac{2}{\theta}} - a^{\frac{1}{\theta}} \le d(\bar{p}, \bar{q}) \le a^{\frac{2}{\theta}} + a^{\frac{1}{\theta}}.
\end{equation}
Observe that $f(p)=g(p)-h(p)$ for all $p\in\mathbb M$, where the components $g$ and $h$ are given by
$$
g(p):=\ a^{2}\ +\ d(p,\bar p)^{2\theta}\ +\ 2b\,d(p,\bar q)^{2\theta}\ +\ 2b\,d(p,\bar p)^{4\theta},\qquad 
h(p):=\ 2a\,d(p,\bar p)^{\theta}\ +\ b\bigl(d(p,\bar p)^{2\theta}+d(p,\bar q)^{\theta}\bigr)^{2}.
$$
By \cite[Thm.~2.1, p.~111]{Udriste1994} we know that for any $z\in\mathbb H^{n}$ and $\alpha\geq 1$ the map $p\mapsto d(p,z)^{\alpha}$ is  geodesically convex. Hence the distance–power terms $d(p,\bar p)^{\theta}$, $d(p,\bar p)^{2\theta}$, $d(p,\bar p)^{4\theta}$, $d(p,\bar q)^{\theta}$, and $d(p,\bar q)^{2\theta}$ are geodesically convex. Moreover, since $t\mapsto t^{2}$ is convex and nondecreasing on $[0,\infty)$, the square of any nonnegative convex function is convex; in particular, $\big(d(p,\bar p)^{2\theta}+d(p,\bar q)^{\theta}\big)^{2}$ is  geodesically convex. By closure of geodesic convexity under nonnegative scaling and addition, both $g$ and $h$ are geodesically convex, and therefore $f=g-h$ is a DC function. 
Since $f(p)\ge 0$ for all $p\in\mathbb H^{n}$ and, by \eqref{eq:defrf}, $f$ in \eqref{eq:defrf} is a sum of squares, any point $p^{*}\in \mathbb{H}^n$ satisfying $d(p^{*},\bar p)^{\theta}=a$ and $d(p^{*},\bar q)^{\theta}=d(p^{*},\bar p)^{2\theta}=a^{2}$ is a global minimizer and satisfies $f(p^{*})=0$. Equivalently, any global minimizer $p^{*}$ must satisfy
$$
d(p^{*},\bar p)=a^{1/\theta},\qquad d(p^{*},\bar q)=a^{2/\theta}.
$$

Conversely, any $p$ satisfying these equalities attains $f(p)=0$ and is a global minimizer. Now, we define the hyperbolic sphere  centered at $\bar{p}$ and $\bar{q}$ as follows 
$$
S(\bar{p}, a^{1/\theta}):= \{ p \in \mathbb{H}^n : d(p, \bar{p}) = a^{1/\theta} \}, \qquad \qquad 
S(\bar{q}, a^{2/\theta}) := \{ p \in \mathbb{H}^n : d(p, \bar{q}) = a^{2/\theta} \}.
$$
Hence, every point $p^*$ in the intersection
$
S(\bar{p}, a^{1/\theta}) \cap S(\bar{q}, a^{2/\theta})
$
is a global minimizer of $f$, with $f(p^*) = 0$. Condition~\eqref{eq:radii-theta} guarantees that this intersection is nonempty. Furthermore, we have
\begin{itemize}
\item If $a^{2/\theta} - a^{1/\theta} < d(\bar{p}, \bar{q}) < a^{2/\theta} + a^{1/\theta}$, then the intersection contains  infinitely many points. Note that if $n=2$ then  just  two points.
\item If $d(\bar{p}, \bar{q}) = a^{2/\theta} - a^{1/\theta}$ (internal tangency) or $d(\bar{p}, \bar{q}) = a^{2/\theta} + a^{1/\theta}$ (external tangency), then the intersection consists of a single point.
\end{itemize}
In fact, \emph{for internal tangency}, choose 
\begin{equation}\label{eq:internal}
\bar{p}_{ _{\mathrm{int}} } = (\sinh a^{1/\theta}, 0, \cosh a^{1/\theta}), \qquad 
\bar{q}_{\mathrm{int}} = (\sinh a^{2/\theta}, 0, \cosh  a^{2/\theta}),
\end{equation}
Using the hyperbolic identity
$
\sinh u \sinh v- \cosh u \cosh v  = -\cosh(u - v),
$
we  obtain that 
$$
\langle \bar{p}_{\mathrm{int}} , \bar{q}_{\mathrm{int}}  \rangle = -\cosh(a^{2/\theta} - a^{1/\theta}),
\quad \text{so} \quad
d(\bar{p}_{\mathrm{int}} , \bar{q}_{\mathrm{int}} ) = \operatorname{arcosh}(\cosh(a^{2/\theta} - a^{1/\theta})) = a^{2/\theta} - a^{1/\theta}.
$$
Thus,
$
S(\bar{p}_{\mathrm{int}} , a^{1/\theta}) \cap S(\bar{q}_{\mathrm{int}} , a^{2/\theta})= \{p^*\}.
$
\emph{For external tangency}, define  $\bar{p}_{\mathrm{ext}}$  and $\bar{q}_{\mathrm{ext}}$ to a new point
\begin{equation}\label{eq:external}
 \bar{p}_{\mathrm{ext}} = (\sinh a^{1/\theta}, 0, \cosh a^{1/\theta}), \qquad \bar{q}_{\mathrm{ext}} = (-\sinh a^{2/\theta}, 0, \cosh  a^{2/\theta}).
\end{equation}
Then,  using the hyperbolic identity
$
 \sinh u \sinh v + \cosh u \cosh v= \cosh(u + v),
$
we conclude that the inner product is
$$
\langle \bar{p}_{\mathrm{ext}}, \bar{q}_{\mathrm{ext}} \rangle = -\sinh a^{1/\theta}  \sinh a^{2/\theta} - \cosh a^{1/\theta}  \cosh a^{2/\theta} = -\cosh(a^{1/\theta} + a^{2/\theta}),
$$
and the hyperbolic distance becomes
$$
d(\bar{p}_{\mathrm{ext}}, \bar{q}_{\mathrm{ext}}) = \operatorname{arcosh}\left( -\langle \bar{p}_{\mathrm{ext}}, \bar{q}_{\mathrm{ext}} \rangle \right) 
= \operatorname{arcosh}\big( \cosh(a^{1/\theta} + a^{2/\theta}) \big) = a^{1/\theta} + a^{2/\theta}.
$$
Thus,  the intersection $S(\bar{p}_{\mathrm{int}} , a^{1/\theta}) \cap S(\bar{q}_{\mathrm{int}} , a^{2/\theta})$ reduces again to a single point, i.e., 
$$
S(\bar{p}_{\mathrm{ext}}, a^{1/\theta}) \cap S(\bar{q}_{\mathrm{ext}}, a^{2/\theta}) = \{p^*\}.
$$
Note that for $\theta > 1$, the function $h$ is continuously differentiable on $\mathbb{H}^n$, whereas for $\theta = 1$, it fails to be differentiable at the points $p = \bar{p}_{\mathrm{ext}}$ and $p = \bar{q}_{\mathrm{ext}}$.
 
\paragraph{Gradient of $h$:}
For $p \ne \bar{p}$ and  $p \ne\bar{q}$,  using the chain rule and \eqref{eq:ddh}, the gradient of $h$ is given by 
\begin{multline}
{\rm grad} h(p) =
-2a \theta\, d(p, \bar{p})^{\theta - 1}\,
\frac{ \langle p, \bar{p} \rangle\, p + \bar{p} }{ \sqrt{ \langle p, \bar{p} \rangle^2 - 1} }
- 4b \theta\, \big( d(p,\bar{p})^{2\theta} + d(p,\bar{q})^{\theta} \big)\,
d(p, \bar{p})^{2\theta - 1}\,
\frac{ \langle p, \bar{p} \rangle\, p + \bar{p} }{ \sqrt{ \langle p, \bar{p} \rangle^2 - 1} }
\\
- 2b \theta\, \big( d(p,\bar{p})^{2\theta} + d(p,\bar{q})^{\theta} \big)\,
d(p, \bar{q})^{\theta - 1}\,
\frac{ \langle p, \bar{q} \rangle\, p + \bar{q} }{ \sqrt{ \langle p, \bar{q} \rangle^2 - 1} }.
\end{multline}

\paragraph{Objective functions for subproblems   ~\eqref{eq:DCAS}  and ~\eqref{eq:B-DCAS}:}
Let $s_k := {\rm grad} h(p_k)$ for $p_k \ne \bar{p}, \bar{q}$. Then, 
\begin{itemize}
\item The objective function for the classical subproblem~\eqref{eq:DCAS} becomes
$$
\psi_k(p)
= a^2 + d(p,\bar{p})^{2\theta} + 2b\,d(p,\bar{q})^{2\theta} + 2b\,d(p,\bar{p})^{4\theta} - \frac{\operatorname{arcosh}(-\langle p_k,p\rangle)}{\sqrt{\langle p_k,p\rangle^2-1}} \langle s_k, p\rangle .
$$
\item The objective function for the Busemann subproblem~\eqref{eq:B-DCAS} becomes
$$
\phi_k(p) = a^2 + d(p,\bar{p})^{2\theta} + 2b\,d(p,\bar{q})^{2\theta}+ 2b\,d(p,\bar{p})^{4\theta} + \|s_k\|\,\ln\!\Big(-\Big\langle p,\; p_k+\frac{s_k}{\|s_k\|}\Big\rangle\Big).
$$
\end{itemize}
In both cases, the smooth convex component $g$ appears explicitly, and the second  component $h$ is ``linearized" either via the exponential map or the Busemann approximation. Here, in test problems with the Rosenbrock-type function we used $\theta = 1$ and $n=2$.   Then, we consider two cases for choosing $\bar{p}, \bar{q} \in \mathbb{H}^2$: internal tangency and external tangency.

For the internal tangency case, we set $a=1$ and $b=100$ and take $\bar{p},\bar{q}$ as in \eqref{eq:internal}. We repeated the test from five random starting points in $\mathbb{H}^2$. 
From Table~\ref{tab:rosen-1}, both CR--DCA and B--DCA display essentially identical behavior: their outer/inner iteration counts coincide up to small fluctuations, and the final objective values and gradient norms agree to numerical accuracy. This parity indicates that the Busemann modeling introduces no observable penalty in this regime, B-DCA is competitive in both accuracy and per-outer effort. While this problem class does not separate the methods, B--DCA retains structural advantages  that can yield steadier progress on more demanding instances.

\begin{table}[H]
\setlength{\tabcolsep}{4.5pt}    
\renewcommand{\arraystretch}{0.9} 
\centering 
\begin{tabular}{lcrrrrrcrrrrr}
\hline 
\ & & \multicolumn{5}{c}{CR-DCA} & & \multicolumn{5}{c}{B-DCA} \\ 
\# & & $k$ & inn & inn$/k$ & fval & grad  & & $k$ inn & inn & inn$/k$ & fval & grad \\ \hline 
1 &  & 158 & 223 & 1.41 & 1.4E-12 & 9.8E-07 &  & 158 & 223 & 1.41 & 1.4E-12 & 9.8E-07 \\
2 &  & 25 & 55 & 2.20 & 8.6E-02 & 2.8E-04 &  & 25 & 55 & 2.20 & 8.6E-02 & 2.8E-04 \\
3 &  & 153 & 217 & 1.42 & 2.2E-12 & 9.4E-07 &  & 153 & 217 & 1.42 & 2.2E-12 & 9.4E-07 \\
4 &  & 148 & 211 & 1.43 & 3.6E-12 & 9.3E-07 &  & 148 & 211 & 1.43 & 3.6E-12 & 9.3E-07 \\
5 &  & 113 & 175 & 1.55 & 1.6E-09 & 1.2E-04 &  & 113 & 175 & 1.55 & 1.6E-09 & 1.2E-04 \\ \hline 
\end{tabular}\caption{Results for Rosenbrock-type objective: internal tangency case.}\label{tab:rosen-1}
\end{table}

For the external tangency case, we set $a=1$ and $b=2$ and take $\bar{p},\bar{q}$ as in \eqref{eq:external}. Using five random starting points in $\mathbb{H}^2$, Table~\ref{tab:rosen-2} shows that B--DCA uses more outer iterations than CR--DCA, which translates into a higher total inner effort on average (mean $k$: $7.42\times 10^3$ vs.\ $6.25\times 10^3$; mean \texttt{inn}: $7.67\times 10^3$ vs.\ $6.43\times 10^3$). Even so, the per-outer cost \texttt{inn}/$k$ is essentially the same in four of the five runs (all between $1.00$ and $1.02$ for both methods), with B--DCA slightly lower in runs~1 and~4; only run~5 is unfavorable to B--DCA ($1.81$ vs.\ $1.49$), which raises its average (mean \texttt{inn}/$k$: $1.17$ vs.\ $1.11$). The final objective values and gradient norms coincide to numerical precision in runs~3--4 and remain very close elsewhere, indicating indistinguishable solution quality. Overall, this table shows that B--DCA is competitive in accuracy and per-outer effort, while CR--DCA is modestly cheaper in total iteration counts on this test.

\begin{table}[H]
\setlength{\tabcolsep}{4.5pt}    
\renewcommand{\arraystretch}{0.9} 
\centering 
\begin{tabular}{lcrrrrrcrrrrr}
\hline 
\ & & \multicolumn{5}{c}{CR-DCA} & & \multicolumn{5}{c}{B-DCA} \\ 
\# & & $k$ & inn &inn$/k$ & fval & grad & &  $k$ & inn &inn$/k$ & fval & grad \\ \hline 
1 &  & 7853 & 8025 & 1.02 & 6.4E-08 & 4.6E-06 &  & 8454 & 8555 & 1.01 & 9.0E-08 & 5.9E-06 \\
2 &  & 4174 & 4238 & 1.02 & 5.3E-07 & 2.8E-05 &  & 4560 & 4659 & 1.02 & 7.4E-07 & 3.6E-05 \\
3 &  & 8989 & 9012 & 1.00 & 2.5E-08 & 2.0E-06 &  & 11300 & 11317 & 1.00 & 2.5E-08 & 2.0E-06 \\
4 &  & 9053 & 9125 & 1.01 & 2.1E-08 & 1.6E-06 &  & 11464 & 11486 & 1.00 & 2.1E-08 & 1.6E-06 \\
5 &  & 1193 & 1773 & 1.49 & 6.5E-05 & 1.7E-03 &  & 1299 & 2347 & 1.81 & 9.1E-05 & 2.2E-03 \\ \hline 
\end{tabular}\caption{Results for Rosenbrock-type objective: external tangency case.}\label{tab:rosen-2}
\end{table}

To close this subsection, we note that across the $\theta$-order hyperbolic Rosenbrock tests (both internal and external tangency, five random starts each), CR--DCA and B--DCA attain essentially the same solution quality, as reflected by the final scaled objective values and gradient norms. B--DCA tends to use more outer iterations on some instances, which can raise the total inner effort, yet its per-outer cost (\texttt{inn}/$k$) remains comparable and is occasionally smaller, consistent with the geodesically convex inner models (with unique minimizers) that it solves at each step. In scenarios where the inner solve dominates the computational cost, this structural feature can be beneficial. Overall, on this class of problems B--DCA is competitive and reliable in accuracy and per-outer effort, while CR--DCA is modestly cheaper in total iteration counts on certain runs.

\subsection{Test problems in positive definite matrices space}
We now turn to the manifold of symmetric positive definite matrices $\mathcal{P}(n)$. Preliminaries on $\mathcal{P}(n)$ and explicit expressions for the Busemann functions and its Riemannian gradient are recalled in Section~\ref{sec:int.2} and Example~\ref{ex:BusSDP}. We consider two set of  problems on $\mathcal{P}(n)$: the first is a synthetic, academically oriented example designed to isolate geometric effects; the second is a practically motivated instance with direct application appeal. The experimental setup follows the same protocol adopted earlier.

\subsubsection{An academic example}
This academic test, also examined in the numerical experiments of \cite[Sec.~7.1]{Bergmann2024}, is designed to validate the implementations of CR–DCA and B-DCA in a controlled setting with known global minimizers and to compare their behavior under the affine–invariant geometry of $\mathcal{P}(n)$. Because the objective depends only on $\log\det X$, the geometric effects appear solely through the subproblems via $\log_{X_k}$ and $d(\cdot,\cdot)$, yielding a clean benchmark. As the results in Table~\ref{tab:power-dca} indicate, both methods reach the global minimum with matching iteration counts, and B-DCA typically achieves lower runtime as $n$ increases, though the trend is not strictly monotonic. It is worth emphasizing that, in this setup, both subproblems are convex; the convexity of the classical subproblem is established in \cite[Example~6.1(i)]{Bergmann2024}. Thus the observed advantage of B-DCA is due to computational effects, its Busemann model yields a well-conditioned tangent-space solve, reducing costly evaluations of $\log_{X_k}(\cdot)$ and $d(\cdot,\cdot)$ and enabling reuse of eigendecompositions across iterations, rather than to any convexity gap.

We first recall some  notations. Throughout, $\ln$ denotes the scalar natural logarithm and ${e}$ the scalar exponential. We consider the difference-of-convex (DC) objective $f=g-h$ on $\mathcal P(n)$ with
$$
g(X)=\bigl(\ln\det X\bigr)^{4}, \qquad h(X)=\bigl(\ln\det X\bigr)^{2}.
$$
 The global minimum value of $f$  is $f^\star=-\tfrac{1}{4}$, attained whenever ${\rm ln} \det X=\pm 1/\sqrt{2}$ (e.g., at $X^\star=e^{\pm 1/(\sqrt{2}\,n)}I_n$).  For the initialization we take
$$
X_0={\rm ln}(n)\,I_n + e_1 e_n^\top + e_n e_1^\top,
$$
where $e_1:=(1, 0 \dots)^T\in {\mathbb R }^{n\times 1}$ and $e_n:=(0,  \dots, 0,1)^T\in {\mathbb R }^{n\times 1}$ , 
and consider dimensions $n\in\{4,10,20,50,100\}$. The stopping criteria and all shared hyperparameters follow our general experimental protocol.  
\paragraph{Gradient of $h$:}
On $\mathcal{P}(n)$  the Riemannian gradient of $h$ and its norm  are  given by 
$$
{\rm grad} h(X) = 2{\rm ln} \!\det(X)\,X,
\qquad 
\|{\rm grad} h(X)\|= 2\sqrt{n}| {\rm ln} \!\det(X)|.
$$
\paragraph{Objective functions for subproblems ~\eqref{eq:DCAS} and ~\eqref{eq:B-DCAS}:}
Let $X_k$ be the current iterate. The objective function of the classical subproblem \eqref{eq:DCAS} reads
$$
\psi_k(X)=\bigl({\rm ln} \det X\bigr)^4-2{\rm ln} \det(X_k){\rm ln} \frac{\det X}{\det X_k}, 
$$
and  the objective function in subproblem \eqref{eq:B-DCAS} is given by  
$$
\phi_k(X):=\bigl({\rm ln} \det X\bigr)^4 + \|{\rm grad} h(X_k)\| B_{X_k, {\rm grad} h(X_k)}(X),
$$
where we use the  explicit  formula \eqref{theo:bus.I.eq} in  Example~\ref{ex:BusSDP}  to compute the Busemann functions $B_{X_k, {\rm grad} h(X_k)}$.

\begin{table}[H]
\setlength{\tabcolsep}{4.5pt}    
\renewcommand{\arraystretch}{0.9} 
\centering
\begin{tabular}{rrrrrrrrrrrrr}
\hline 
 & & \multicolumn{5}{c}{CR-DCA} & & \multicolumn{5}{c}{B-DCA} \\
$n$ & & $k$ & inn & time & fval & grad & &  $k$ & inn & time & fval & grad \\ \hline 
4 &  & 13 & 25 & 0.49 & -0.25 & 7.52E-07 &  & 13 & 25 & 0.43 & -0.25 & 7.52E-07 \\
10 &  & 16 & 43 & 0.53 & -0.25 & 5.09E-07 &  & 16 & 43 & 0.34 & -0.25 & 5.09E-07 \\
20 &  & 16 & 49 & 0.11 & -0.25 & 1.01E-06 &  & 16 & 49 & 0.08 & -0.25 & 1.01E-06 \\
50 &  & 16 & 55 & 0.26 & -0.25 & 2.13E-06 &  & 16 & 55 & 0.14 & -0.25 & 2.13E-06 \\
100 &  & 16 & 59 & 1.03 & -0.25 & 3.55E-06 &  & 16 & 59 & 0.46 & -0.25 & 3.55E-06 \\ \hline 
\end{tabular}
\caption{Results in the academic problem for CR-DCA and B-DCA.}\label{tab:power-dca}
\end{table}
Table~\ref{tab:power-dca} reports the outcomes for CR–DCA and B–DCA. We observe that both algorithms display identical outer/inner iteration counts across all tested dimensions and attain the same objective value $f^\star=-0.25$ to the reported precision, with final gradient norms $\le 3.6\times 10^{-6}$. As $n$ increases from $4$ to $100$, the number of outer iterations stabilizes at $k=16$, while the inner iterations grow moderately (from $25$ to $59$). In terms of runtime, B–DCA is consistently faster, with approximate reductions of $12\%$, $36\%$, $27\%$, $46\%$, and $55\%$ for $n=4,10,20,50,100$, respectively. The advantage tends to increase with dimension (though not strictly monotonically), indicating a lower per–iteration overhead in the Busemann-subgradient subproblem on $\mathcal{P}(n)$.

\subsubsection{Contrastive learning via DC optimization in $\mathcal{P}(n)$} \label{sec:contrastive-SPD}

In this section we consider a contrastive DC optimization model on the manifold of symmetric positive definite matrices $\mathcal{P}(n)$. Given disjoint sets of reference points $\mathfrak{P}=\{\bar X_1,\ldots,\bar X_m\}\subset\mathcal{P}(n)$ and $\mathfrak{N}=\{\bar Y_1,\ldots,\bar Y_r\}\subset\mathcal{P}(n)$, the goal is to select $X\in\mathcal{P}(n)$ that is close to $\mathfrak{P}$ and far from $\mathfrak{N}$, where proximity is measured by the affine–invariant geodesic distance $d(\cdot,\cdot)$ defined in~\eqref{eq:RiemannianDistance}. The \emph{SPD contrastive objective} is
\begin{equation}\label{eq:contrastive-spd}
f(X) := \sum_{i=1}^{m}\lambda_i^{+}\, d^2\!\big(X,\bar X_i\big)\ -\ \sum_{j=1}^{r}\lambda_j^{-}\, d^2\!\big(X,\bar Y_j\big),
\qquad X\in\mathcal{P}(n),
\end{equation}
with fixed weights $\lambda_i^+,\lambda_j^->0$. Since $X\mapsto d^2(X,\bar Z)$ is geodesically strongly convex on $\mathcal{P}(n)$ for each fixed $\bar Z$, $f$ admits the DC splitting
\begin{equation}\label{eq:g-h-contrastive-spd}
g(X):=\sum_{i=1}^{m}\lambda_i^{+}\, d^2\!\big(X,\bar X_i\big),\qquad \qquad 
h(X):=\sum_{j=1}^{r}\lambda_j^{-}\, d^2\!\big(X,\bar Y_j\big),
\end{equation}
so that $f=g-h$. This objective encodes the contrastive principle: $g$ promotes proximity to positives, while $h$ penalizes proximity to negatives. Such formulations are natural when data are represented by covariance or kernel matrices, e.g., in signal processing, computer vision, and Riemannian manifold learning; see \cite{bao2024symcl,Tibermacine2024}.

We use this example to contrast the inner models underlying CR–DCA and B–DCA on $\mathcal{P}(n)$. The first–order model of $h$ used by CR–DCA does not, in general, preserve geodesic convexity under the affine–invariant metric. In B–DCA, $h$ is replaced by a Busemann-based surrogate; under our convention that the Busemann functions is concave (hence $-B$ is convex), the inner objective is geodesically convex. Accordingly, the reported performance differences should be read through this convex B–DCA versus nonconvex CR–DCA contrast and its computational implications (conditioning and reuse of eigendecompositions).

Our objectives are twofold: (i) to verify that CR–DCA and B–DCA compute meaningful contrastive representatives $X^\star\in\mathcal{P}(n)$ with comparable solution quality; and (ii) to compare their numerical behavior under the affine–invariant geometry. We report final objective values, Riemannian gradient norms, iteration counts, and runtime, emphasizing the impact of the inner–subproblem model on efficiency. The experiments use $n=10$ with unit weights $\lambda_i^{+}=\lambda_j^{-}=1$ and follow the stopping criteria and shared hyperparameters of our general experimental protocol.

\paragraph{Objective functions for subproblems~\eqref{eq:DCAS} and~\eqref{eq:B-DCAS}.}
Let $ S_k := {\rm grad} h(X_k) $, where $ X_k \in \mathcal{P}(n) $ is the current iterate. Then, the subproblem objectives are:
\begin{itemize}
    \item For the classical DC subproblem~\eqref{eq:DCAS}, the objective becomes:
    \begin{equation}
    \Psi_k(X) := \sum_{i \in \mathfrak{P}} \lambda_i^+\, d(X, \bar{X}_i)^2 - \langle S_k, \log_{X_k} X \rangle,
    \end{equation}
    where the inner product is as in~\eqref{eq:AffineInvariantMetric} and $ \log_{X_k} X $ is given by~\eqref{eq:LogSPD}.
    \item For the Busemann-regularized DC subproblem~\eqref{eq:B-DCAS}, the objective becomes:
    \begin{equation}
    \Phi_k(X) := \sum_{i \in \mathfrak{P}} \lambda_i^+\, d(X, \bar{X}_i)^2 + \|S_k\|  B_{{P_k}, S_k}(X),
    \end{equation}
 where  we use the  explicit  formula \eqref{theo:bus.I.eq} in  Example~\ref{ex:BusSDP}  to compute the Busemann functions $B_{{P_k}, S_k}(X)$
\end{itemize}

\begin{table}[H]
\setlength{\tabcolsep}{4.5pt}    
\renewcommand{\arraystretch}{0.9} 
\centering 
\begin{tabular}{lcrrrrrcrrrrr}
\hline 
\ & & \multicolumn{5}{c}{CR-DCA} & & \multicolumn{5}{c}{B-DCA} \\ 
\# & & $k$ & inn & inn/$k$ & fval & grad & &  $k$ & inn & inn/$k$ & fval & grad \\ \hline 
1 &  & 7 & 13 & 1.86 & 0.1663 & 1.14E-05 &  & 8 & 16 & 2.00 & 0.1663 & 6.74E-06 \\
2 &  & 7 & 14 & 2.00 & 0.1981 & 1.08E-05 &  & 8 & 15 & 1.88 & 0.1981 & 4.88E-06 \\
3 &  & 7 & 14 & 2.00 & 0.1520 & 1.13E-05 &  & 8 & 15 & 1.88 & 0.1520 & 4.98E-06 \\
4 &  & 7 & 14 & 2.00 & 0.1825 & 1.12E-05 &  & 8 & 16 & 2.00 & 0.1825 & 6.62E-06 \\
5 &  & 7 & 13 & 1.86 & 0.2071 & 1.07E-05 &  & 8 & 15 & 1.88 & 0.2071 & 4.98E-06 \\
6 &  & 7 & 14 & 2.00 & 0.2252 & 1.06E-05 &  & 8 & 15 & 1.88 & 0.2252 & 4.95E-06 \\
7 &  & 7 & 13 & 1.86 & 0.1960 & 1.09E-05 &  & 8 & 15 & 1.88 & 0.1960 & 5.81E-06 \\
8 &  & 7 & 14 & 2.00 & 0.1762 & 1.13E-05 &  & 8 & 16 & 2.00 & 0.1762 & 6.10E-06 \\
9 &  & 7 & 14 & 2.00 & 0.1686 & 1.11E-05 &  & 8 & 15 & 1.88 & 0.1686 & 4.85E-06 \\
10 &  & 7 & 13 & 1.86 & 0.1918 & 1.10E-05 &  & 8 & 15 & 1.88 & 0.1918 & 5.78E-06 \\ \hline 
\end{tabular}\caption{Results for contrastive learning in $\mathcal{P}(5)$, with $m=5$ and $r=1$.}\label{tab:contra-spd-1}
\end{table}

Table~\ref{tab:contra-spd-1} summarizes ten runs with random initializations in $\mathcal{P}(5)$ with  $m=5$ and  $r=1$. Both methods attain identical objective values across all restarts, confirming comparable solution quality. A clear difference appears in stationarity: B–DCA yields strictly smaller final Riemannian gradient norms in every trial (mean $5.569\times 10^{-6}$) than CR–DCA (mean $1.103\times 10^{-5}$), i.e., an almost twofold reduction. This tighter stationarity can naturally entail a slightly larger overall inner effort: B–DCA uses one additional outer iteration ($k=8$ versus $k=7$) and, consequently, a higher average total of inner iterations (means $15.3$ versus $13.6$). At the same time, the inner burden per outer step is comparable, if anything, marginally lower for B–DCA, as reflected by the average inn/$k$ ratios ($1.9125$ for B–DCA versus $1.9429$ for CR–DCA) and medians (15 vs.\ 14). Overall, this experiment highlights B–DCA’s advantage of achieving stronger first–order stationarity while keeping the inner workload per outer step essentially on par, consistent with the convexity of its inner model.

\begin{table}[H]
\setlength{\tabcolsep}{4.5pt}    
\renewcommand{\arraystretch}{0.9} 
\centering 
\begin{tabular}{lcrrrrrcrrrrr}
\hline 
\ & & \multicolumn{5}{c}{CR-DCA} & & \multicolumn{5}{c}{B-DCA} \\ 
\# & & $k$ & inn & inn/$k$ & fval & grad & &  $k$ & inn & inn/$k$ & fval & grad \\ \hline 
1 &  & 51 & 91 & 1.78 & -0.6747 & 2.03E-05 &  & 83 & 129 & 1.55 & -0.6747 & 2.12E-05 \\
2 &  & 50 & 91 & 1.82 & -0.6587 & 2.14E-05 &  & 81 & 125 & 1.54 & -0.6587 & 1.99E-05 \\
3 &  & 49 & 87 & 1.78 & -0.7142 & 2.27E-05 &  & 80 & 124 & 1.55 & -0.7142 & 2.27E-05 \\
4 &  & 50 & 90 & 1.80 & -0.6460 & 1.82E-05 &  & 78 & 122 & 1.56 & -0.6460 & 2.06E-05 \\
5 &  & 53 & 96 & 1.81 & -0.6116 & 1.72E-05 &  & 85 & 136 & 1.60 & -0.6116 & 1.97E-05 \\
6 &  & 52 & 94 & 1.81 & -0.6051 & 1.73E-05 &  & 83 & 130 & 1.57 & -0.6051 & 1.85E-05 \\
7 &  & 49 & 85 & 1.73 & -0.6570 & 1.80E-05 &  & 71 & 112 & 1.58 & -0.6570 & 1.99E-05 \\
8 &  & 53 & 97 & 1.83 & -0.5740 & 1.79E-05 &  & 88 & 140 & 1.59 & -0.5740 & 1.87E-05 \\
9 &  & 53 & 96 & 1.81 & -0.5892 & 1.73E-05 &  & 87 & 139 & 1.60 & -0.5892 & 1.93E-05 \\
10 &  & 49 & 88 & 1.80 & -0.6763 & 2.21E-05 &  & 75 & 118 & 1.57 & -0.6763 & 2.13E-05 \\ \hline 
\end{tabular}\caption{Results for contrastive learning in $\mathcal{P}(5)$, with $m=5$ and $r=4$.}\label{tab:contra-spd-2}
\end{table}

Table~\ref{tab:contra-spd-2} summarizes ten runs with random initializations in $\mathcal{P}(5)$ with  $m=5$, $r=4$. Both methods deliver indistinguishable solution quality: objective values match entrywise, and the final Riemannian gradient norms are of the same order ($\approx 2\times 10^{-5}$) for both algorithms. The principal contrast lies in efficiency. B–DCA requires more outer iterations (mean $81.1$ vs.\ $50.9$ for CR–DCA), but its inner subproblems are consistently easier: the inner-per-outer ratio $\texttt{inn}/k$ is strictly smaller for B–DCA in all ten trials (average $1.572$ vs.\ $1.797$, about $12.5\%$ lower). 
The pattern is precisely what the modeling predicts: B–DCA’s convex Busemann-based inner model yields cheaper inner solves (and facilitates eigendecomposition reuse), so the extra outer iterations are offset by a reduced per-iteration burden.

To close this section, we synthesize the evidence across the three groups of tests:  (i) the academic benchmark on $\mathcal{P}(n)$, (ii) the contrastive learning tasks on $\mathcal{P}(n)$, and (iii) the $\theta$-order hyperbolic Rosenbrock problems on $\mathbb{H}^2$. Although the surrogate objectives in \eqref{eq:Sugdlhs} and \eqref{eq:Sugbf} are analytically distinct, they track each other closely along the observed iterates, which plausibly explains the near-identical solution quality achieved by CR--DCA and B--DCA in all experiments (final objective values and first-order stationarity). The main practical difference lies in the structure of the inner problems: the Busemann model \eqref{eq:B-DCAS} is geodesically convex, and this often translates into steadier progress and a comparable or lower per-outer effort (\texttt{inn}/$k$), even in cases where B--DCA uses more outer iterations.

On $\mathcal{P}(n)$, the academic test shows matching iteration counts and solution quality, with B--DCA frequently exhibiting favorable runtime as $n$ increases, consistent with concentrating the geometry in well-conditioned tangent-space solves (fewer expensive evaluations of $\log_{X_k}(\cdot)$ and $d(\cdot,\cdot)$ and better reuse of eigendecompositions). In the contrastive learning experiments, the methods again deliver essentially the same accuracy; B--DCA often requires more outer steps but displays lower \texttt{inn}/$k$ and more regular progress across restarts, an advantage that becomes relevant when inner solves dominate the computational budget. For the hyperbolic Rosenbrock family, the internal tangency tests show virtually identical behavior, and the external tangency tests are mixed: CR--DCA is sometimes cheaper in total counts, whereas the per-outer cost remains comparable and the attained solutions are indistinguishable in quality, so B--DCA remains competitive.

Taken together, these results present CR--DCA as a strong baseline and B--DCA as a competitive alternative that can be particularly attractive when (i) inner subproblem cost dominates runtime, (ii) numerical stability of the inner model matters, or (iii) one wishes to exploit  the geodesic convexity  in \eqref{eq:B-DCAS} for steadier progress. Beyond these experiments, the Busemann modeling and the side-by-side comparison of \eqref{eq:Sugdlhs} and \eqref{eq:Sugbf} introduce new ideas into DC optimization on manifolds and suggest avenues to refine convergence theory for DCA-type schemes. 

\section{Conclusions} \label{sec:Conclusions}

In summary, this paper investigates fundamental properties of Busemann functions on Hadamard manifolds and highlights their role in the design and analysis of optimization algorithms on Riemannian manifolds. This is achieved through the introduction of a novel Busemann-based characterization of the classical subdifferential in Riemannian optimization. The proposed approach addresses challenges arising from nonconvex subproblem functions on Hadamard manifolds, which commonly occur in classical difference-of-convex algorithms. By replacing the inner product with Busemann functions, the resulting reformulation guarantees the strong convexity of the subproblem functions, thereby improving both optimization performance and algorithmic efficiency in the Hadamard manifold setting.

We are confident that the characterization developed here for DC problems will also be useful in broader areas of continuous optimization. For instance, in an upcoming paper, we intend to employ it to develop algorithms utilizing the Bregman distance concept. To illustrate, let us explore how we can establish the Bregman distance concept using the Busemann functions, which forms the basis of our discussion: Let $\mathbb{M}$ be a Hadamard manifold and $S \subseteq \mathbb{M}$ be an open and convex set  and ${\bar S}$ its closure. Consider a {\it proper  convex real  function  $\psi: {\bar S} \rightarrow \mathbb{R} \cup \{+\infty\}$}, which is differentiable on the set $S$,  and let  $D_{\psi}(\cdot,\cdot): {\bar S}  \times {S}  \rightarrow \mathbb{R} \cup \{+\infty\}$ be a  function   associated to $\psi$ defined by
\begin{equation} \label{eq:deBFDef}
D_{\psi}(p, q) :=   \psi(p)-\psi(q)+ \|{\rm grad} \psi(q)\|B_{q, {\rm grad} \psi(q)}(p), 
\end{equation}
where $B_{q, {\rm grad} \psi(q)}$ is the Busemann functions  with a base point  ${q} \in  {\mathbb M}$  and  associated direction  ${ {\rm grad} \psi(q)} \in T_{q} {\mathbb M}$.  To introduce the subsequent definition, let us define the notation for the {\it partial level sets of $D_{\psi}$} as follows: for any  $\alpha \in \mathbb{R}$, consider
\begin{equation*}
\mathcal{L}_1(\alpha,{q}) := \{p \in {\bar S}:~D_{\psi}(p,q) \leq \alpha\},\qquad \quad  \mathcal{L}_2(p,\alpha) := \{q \in {S}:~ D_{\psi}(p, q) \leq \alpha\}.
\end{equation*}
The extension of the Bregman distance concept to Hadamard manifolds utilizing a Busemann functions is as follows:
\begin{definition} \label{def:bf} 
The function $\psi$ is called a {\it Bregman-Busemann  function} and $D_{\psi}$ a {\it Bregman-Busemann distance induced by $\psi$} if the following conditions hold:
\begin{enumerate}
    \item[$(i)$] $\psi$ is continuously differentiable on $S$;
    \item[$(ii)$] $\psi$ is strictly convex and  continuous on ${\bar S}$;
    \item[$(iii)$] For all $\alpha \in \mathbb{R}$ the partial level sets $\mathcal{L}_1(\alpha, q)$ and $\mathcal{L}_2(p, \alpha)$ are bounded for every $q \in S$ and $p \in {\bar S}$, respectively.
   \item[$(iv)$]  If $(q_k)_{k\in {\mathbb N}}\subset S$ converges to ${\bar q}$ then $D_{\psi}({\bar q},q_k)$ converges to $0$.
    \item[$(v)$]  If $(p_k)_{k\in {\mathbb N}}\subset {\bar S}$ and $(q_k)_{k\in {\mathbb N}}\subset S$  are sequences such that $(p_k)_{k\in {\mathbb N}}$ is bounded, $\lim q_k = {\bar q}$, and $\lim_{k \to \infty} D_{\psi}(p_k,q_k)= 0$, then $\lim p_k = {\bar q}$.
\end{enumerate}
\end{definition}
We proceed with some comments regarding Definition~\ref{def:bf}.  The set $S$ is referred to as the {\it zone} of $\psi$.  Utilizing \eqref{eq:defFIWDef2b}, we deduce that $(iv)$ and $(v)$ hold when ${\bar q} \in S$, as a consequence of (i), (ii), and (iii), thereby necessitating their verification solely at points on the boundary $\partial S$ of $S$. In the following proposition, we demonstrate that the Bregman-Busemann distance is indeed a convex distance in the first coordinate.
\begin{proposition} \label{le;ipbbf}
Let $\psi$ be a Bregman-Busemann  function with zone $S$ and $D_{\psi}$ be  the Bregman-Busemann distance induced by $\psi$. Then, the following statements hold:
\begin{enumerate}
     \item[(i)] $D_{\psi}(p, q)\geq 0$, for all $p \in {\bar S}$ and $q \in S$;
     \item[(ii)]   $D_{\psi}(\cdot, q): {\bar S}   \rightarrow \mathbb{R} \cup \{+\infty\}$ is strictly convex, for all $q\in {S}$;
\end{enumerate}
\end{proposition}
\begin{proof}
If $\psi$ and $D_{\psi}$ satisfy Definition~\ref{def:bf}, then it follows from Theorem~\ref{th:ebesub}  that function \eqref{eq:deBFDef} satisfies item $(i)$. It follows from Lemma~\ref{le:CharactBusFunc} that $B_{q, {\rm grad} \psi(q)}$  is convex. Thus, by using item $(ii)$ of Definition~\ref{def:bf} the  proof of item $(ii)$ follows.
\end{proof}

The Bregman distance, introduced as a fundamental concept in \cite{Bregma1967}, has been extensively studied in the Euclidean context, as evidenced by various references including \cite{Bauschke1997, ChenTeboulle1993, dePierroIusem1996, Kiwiel1997}. The idea of developing the Bregman-Busemann distance is inspired by Example~\ref{ex:defBFNDefK0}, which illustrates how our Definition~\ref{def:bf} serves as a natural extension of the established concept of the Bregman distance in Euclidean spaces. Furthermore, it is noteworthy that the concept of Bregman functions has been previously introduced in the context of Hadamard manifolds, utilizing the function \eqref{eq:Sugdlhs} in its definition, as mentioned in \cite{PapaQuiro2009, PapaQuiro2013}. However, this definition results in a Bregman function that lacks convexity in the first coordinate. This limitation has been acknowledged in the papers \cite{PapaQuiro2009, PapaQuiro2013}, where it led to erroneous outcomes, thereby restricting its applicability. It is important to highlight that the introduction of the Bregman-Busemann distance in Definition~\ref{def:bf} addresses this limitation, as demonstrated in Proposition~\ref{le;ipbbf}.

\section{Appendix} \label{Appendix}

\begin{appendixthm}\label{app:A3}
\begin{proof}[Busemann functions on hyberbolic space in Example~\ref{ex:BusHype}:] To simplify the notations we  set $\alpha(t):=-\kappa \left\langle p ,  \exp_{q}(tv)\right\rangle$ with $v\neq 0$. Since
$\left\langle p , q\right\rangle\leq -1/\kappa$, for all  $p, q\in {\mathbb
H}^{n}_{\kappa}$, we have $\alpha(t)\geq 1$. Thus,    ${\rm arcosh}
(\alpha(t))=\ln\left(\alpha(t)+\sqrt{\alpha(t)^2-1}\right)$, or equivalently, 
\begin{equation} \label{eq:mdr}
{\rm arcosh} (\alpha(t))=\ln\left(\alpha(t)(1+\beta(t))\right), \quad \quad
\mbox{where}~  \beta(t):=\left({1-{(1}/{\alpha(t))^2}}\right)^{1/2}.
\end{equation}
By using \eqref{eq:Intdist}  we  obtain  that  $d_{\kappa}(p,  \exp_{q}(tv))=({1}/{\sqrt{\kappa}}){\rm arcosh} (\alpha(t))$.  Hence, taking into account that  $({1}/{\sqrt{\kappa}})\ln e^{-{\sqrt{\kappa}}\|v\|t}=-\|v\|t$, it follows from \eqref{eq:mdr} that 
\begin{equation} \label{eq:Intdistga}
d_{\kappa}(p,  \exp_{q}(tv))-\|v\|t=\frac{1}{\sqrt{\kappa}}\ln\left(e^{-{\sqrt{\kappa}}\|v\|t}\alpha(t)\left(1+\beta(t)\right)\right).
\end{equation}
Due to $\alpha(t):=-\kappa \left\langle p ,  \exp_{q}(tv)\right\rangle$, by the definitions of  $ \cosh$ and  $\sinh$ we  obtain  that 
\begin{align*}
\alpha(t)&= -\kappa \frac{1}{2}\left(e^{\sqrt{\kappa}\,\|v\|t}+ e^{-\sqrt{\kappa}\,\|v\|t}\right) \left\langle p, {q}\right\rangle- \sqrt{\kappa}  \frac{1}{2}\left(e^{\sqrt{\kappa}\,\|v\|t}- e^{-\sqrt{\kappa}\,\|v\|t}\right) \frac{1}{\|v\|}\left\langle p, {v}\right\rangle \\
             &= -\frac{1}{2}e^{\sqrt{\kappa}\,\|v\|t}\left(\kappa \left( 1+   e^{-2\sqrt{\kappa}\,\|v\|t}\right) \left\langle p, {q}\right\rangle+ \sqrt{\kappa} \left( 1-   e^{-2\sqrt{\kappa}\,\|v\|t}\right) \frac{1}{\|v\|}\left\langle p, {v}\right\rangle \right).
\end{align*}
Hence, we  have   $\lim_{t \to +\infty}  \alpha(t)^2=+\infty,$
$
\lim_{t \to \infty}  e^{-{\sqrt{\kappa}}\|v\|t}\alpha(t)
= -\frac{1}{2}\left(\kappa\left\langle p, {q}\right\rangle+\sqrt{\kappa}
\frac{1}{\|v\|}\left\langle p, {v}\right\rangle \right)
$ and $\lim_{t \to \infty} \beta(t)=1$.  Thus, we conclude from \eqref{eq:Intdistga} that 
\begin{equation*}
 \lim_{t \to+ \infty} \left(d\left({p}, \exp_{q}(tv)\right)-\|v\|t\right)=
 \frac{1}{\sqrt{\kappa}}\ln\left(- \left\langle p, \kappa\,{q}+
 {\sqrt{\kappa}}\,\frac{1}{\|v\|}{v}\right\rangle \right).
\end{equation*}
Therefore, the last equality together with \eqref{eq:defBFNDef} gives \eqref{eq:BFkSF}  the desired equality. Now, we are going to compute the gradient of  Busemann functions. It follows from \eqref{eq:grad} that 
\begin{equation} \label{eq:gradn}
{\rm grad} B_{{q}, v}(p):= {\rm Proj}^{\kappa}_p {\rm J}B_{{q}, v}'(p)=  {\rm J}B_{{q}, v}'(p)+ \kappa \left\langle {\rm J}B_{{q}, v}'(p), p\right\rangle p.
\end{equation}
To simplify the notations we set  $B_{{q}, v}(p)
:= ({1}/{\sqrt{\kappa}})\ln (\eta(p))$, where $\eta(p):= -\left\langle p,
\kappa\,{q}+ {\sqrt{\kappa}}\frac{1}{\|v\|}{v}\right\rangle$.
Thus, taking the Euclidean derivative  we have
\begin{equation} \label{eq:comBFed}
	B_{{q}, v}'(p)=\frac{1}{\sqrt{\kappa}}\frac{\eta'(p)}{\eta(p)}, \qquad
	\quad  \eta'(p)= -{\rm J}\left(\kappa\,{q}+
	{\sqrt{\kappa}}\frac{1}{\|v\|}{v}\right).
\end{equation}
Substituting the last equalities \eqref{eq:comBFed} into \eqref{eq:gradn}, after some algebraic manipulations,   we obtain  that 
\begin{equation*}
	{\rm grad} B_{{q}, v}(p)=\frac{1}{\sqrt{\kappa}}\frac{1}{\eta(p)}\left({\rm J}\eta'(p)+ \kappa \left\langle {\rm J}\eta'(p), p\right\rangle \,p \right).
\end{equation*}
Substituting \eqref{eq:comBFed} into the last  equality and considering  that
${\rm J}{\rm J}=I$,  we obtain  \eqref{eq:grad2}. Note that some calculations show that \eqref{eq:grad2} implies that $\|{\rm grad} B_{{q}, v}(p)\|=1$, as stated in Lemma~\ref{le:CharactBusFunc}. Particularly, if $p=q$, then due to  $\langle q, q\rangle=-{1}/{\kappa}$ and $\langle p, v \rangle=0$, the final equation simplifies to
\begin{equation*} 
	{\rm grad} B_{{q}, v}(q)=-\frac{1}{\sqrt{\kappa}}\left(\kappa\,{q}+
	{\sqrt{\kappa}}\frac{1}{\|v\|}{v}- \kappa q
\right)=-\frac{1}{\|v\|}{v}, 
\end{equation*}
which is in accordance with the last statement of Lemma~\ref{le:CharactBusFunc}.
\end{proof}
\end{appendixthm}

\begin{appendixthm}\label{app:A4}
We now present the detailed computation of the explicit formula given in Example~\ref{ex:BusSDP} for the Busemann functions and its gradient on the manifold of symmetric positive definite matrices. It is worth noting that this computation does not require any prior knowledge of the theory of symmetric spaces. To this end, we first prove two auxiliary lemmas.
\begin{lemma}\label{theo:main.commute}
Take $ X \in  \mathcal{P}(n) $ and $ V\in \mathcal{S}(n) $. If $X$ commutes with $V$, then 
$$
B_{I , V}(X)   = - \frac{1}{\|V\|} {\rm tr} \left( V {\rm Log}\left( X \right) \right).
$$
\end{lemma}
\begin{proof}
Since $X$ commutes with $V$, applying Proposition~\ref{pr:afcbf} and equation~\eqref{eq:RiemannianDistance}, and after some algebraic manipulations, we obtain
\begin{align*}
&B_{I, V}(X)  = \frac{1}{2\|V\|  } \lim_{t \to+ \infty}   \frac{ d^2(  {\rm Exp}(tV) , X ) - ( t\|V\| )^2 }{t}\\
                   & =   \frac{1}{2\|V\|  } \lim_{t \to+ \infty} \frac{1}{t} {\rm tr} \big(  \big[ {\rm Log} ( X^{-1/2}   {\rm Exp}(t V)  X^{-1/2})  \big]^2 - \left( t V  \right)^2 \big)\\
                  & =  \frac{1}{2\|V\|} \lim_{t \to+ \infty}   \frac{1}{t} {{\rm tr} \big(\big[  {\rm Log} ( X^{-1/2}    {\rm Exp}(tV)   X^{-1/2}) - t V \big] \big[     {\rm Log} (X^{-1/2}   {\rm Exp}(tV)  X^{-1/2}) + t V  \big] \big) }.
\end{align*}
Since $X$ commutes with $V$, we have  $ {\rm Log} ( X^{-1/2}   {\rm Exp}(tV)  X^{-1/2})={\rm Log}({\rm Exp}(tV))- {\rm Log}(X)$. Thus, the last equality can be written as 
\begin{align*}
B_{I, V}(X) 
  &= \frac{1}{2\|V\|}  \lim_{t \to+ \infty}   \frac{1}{t}{{\rm tr} \left(  \left[  {\rm Log}\left(   {\rm Exp}(tV)    \right) -   {\rm Log}\left( X  \right) -  t V   \right] \left[ {\rm Log}\left( {\rm Exp}(tV)    \right) -   {\rm Log}\left( X  \right) +  t V  \right] \right) }\\
& = \frac{1}{2\|V\|}  \lim_{t \to+ \infty}   \frac{1}{t}{\rm tr} \left(\left[-{\rm Log}\left( X  \right)  \right] \left[ 2t V - {\rm Log}\left( X  \right)  \right] \right) = - \frac{1}{\|V\|}       {\rm tr} \left(   V {\rm Log}\left( X  \right)   \right),
\end{align*}
which is the desired equality.
\end{proof}
Let  $ L $ be  the  lower triangular matrix  given in Example~\ref{ex:BusSDP}.  For our analysis it is convenient to use the decomposition of the form $L= W Z$, with 
\begin{equation} \label{eq:fmatrces}
W := \begin{bmatrix}
I_{n_1} & 0 & \cdots & 0 \\
W_{2 1} & I_{n_2} & \cdots & 0 \\
\vdots & \vdots & \ddots & \vdots \\
W_{k1} & W_{k2} & \cdots & I_{n_k}
\end{bmatrix}
\qquad  \qquad 
Z := \begin{bmatrix}
L_{n_1} & 0 & \cdots & 0 \\
0& L_{n_2} & \cdots & 0 \\
\vdots & \vdots & \ddots & \vdots \\
0 & 0 & \cdots & L_{n_k}
\end{bmatrix}.
\end{equation} 
\begin{lemma}\label{lem:aux.main.theo}
Let  $D$  and $L$ matrices as defined  in Example~\ref{ex:BusSDP}, and  $W$ and   $Z$  matrices  as defined in \eqref{eq:fmatrces}. Thus, the following equalities   hold:
\begin{enumerate}
\item[(i)] $\lim_{t\to +\infty} d\big( W^{-1} {\rm Exp}(t D) [W^{T}]^{-1}, {\rm Exp}(t D) \big)  = 0$;
\item[(ii)] $\lim_{t \to +\infty} d\big(  L^{-1}  {\rm Exp}\left( t   D  \right) \left( L^{T} \right)^{-1} , {\rm Exp}\left( tD \right) \big) = d(I,ZZ^T)$.
\end{enumerate}
\end{lemma}
\begin{proof}
We prove item~(ii) only, as item~(i) is analogous. The second part of Lemma~\ref{lem:Log.property} implies that 
\begin{equation} \label{eq:imme}
d\big(L^{-1}  {\rm Exp}\left( t   D  \right) \left( L^{T} \right)^{-1} , {\rm Exp}\left( tD \right) \big)  = d\big(  I  , {\rm Exp}^{-1/2} \left( tD \right) L {\rm Exp}(tD) L^{T} {\rm Exp}^{-1/2} ( tD)\big).
\end{equation}
To simplify the notation  set  $A(t):= {\rm Exp}^{-1/2} \left( tD \right)  L  {\rm Exp}^{1/2} \left( tD \right)$. Note that $A(t)$ can be written as 
$$
A(t) :=\begin{bmatrix}
L_{n_1} & 0 & \cdots & 0 \\
e^{\frac{t(\lambda_1 - \lambda_2)}{2}} L_{21} & L_{n_2} & \cdots & 0 \\
\vdots & \vdots & \ddots & \vdots \\
e^{\frac{t(\lambda_1 - \lambda_k)}{2}} L_{k1} & e^{\frac{t(\lambda_2 - \lambda_k)}{2}} L_{k2} & \cdots & L_{n_k}
\end{bmatrix}.
$$
Thus, equality \eqref{eq:imme}  is equivalent to 
$
d\big(L^{-1}  {\rm Exp}\left( t   D  \right) \left( L^{T} \right)^{-1} , {\rm Exp}\left( tD \right)\big)= d\big( I  ,  A(t) A(t)^{T} \big).
$
Since $ \lambda_1 < \dots < \lambda_k $, we can conclude that $\lim_{t \to +\infty} A(t) = Z$, which completes the proof.
\end{proof}
\begin{proof}[Busemann functions on $\mathcal{P}(n)$ in Example~\ref{ex:BusSDP}:] 
Let the spectral and Cholesky factorizations be
\begin{equation} \label{eq:ccfb}
    Y^{-1/2} V Y^{-1/2}=U D U^{T}, \qquad U^{T} Y^{-1/2} X Y^{-1/2} U=L L^{T}=W Z Z^{T} W^{T},
\end{equation} 
with $L=WZ$.  Since
$
    {\rm exp}_{Y}(tV)
      =Y^{1/2}{\rm Exp}\!\bigl(tY^{-1/2} V Y^{-1/2}\bigr)Y^{1/2}
      =Y^{1/2}U\,{\rm Exp}(tD)\,U^{T}Y^{1/2},
$
 Lemma~\ref{lem:Log.property} and \eqref{eq:ccfb} imply that
\begin{align}\label{eq:BusSDP:p1}
d\bigl({\rm exp}_{Y}(tV),X\bigr)
   &= d\bigl({\rm Exp}(tD),U^{T}Y^{-1/2} X Y^{-1/2}U\bigr)\nonumber\\
   &= d\bigl({\rm Exp}(tD),W Z Z^{T} W^{T}\bigr)\nonumber\\
   &= d\bigl(W^{-1}{\rm Exp}(tD)[W^{T}]^{-1}, Z Z^{T}\bigr).
\end{align}
On the other hand, by applying the triangle inequality, we obtain
\begin{align}\label{eq:BusSDP:p2}
&d\bigl({\rm Exp}(tD),Z Z^{T}\bigr)
 -d\bigl(W^{-1}{\rm Exp}(tD)[W^{T}]^{-1},{\rm Exp}(tD)\bigr)\nonumber\\
&\le d\bigl(W^{-1}{\rm Exp}(tD)[W^{T}]^{-1},Z Z^{T}\bigr)\\
&\le d\bigl({\rm Exp}(tD),Z Z^{T}\bigr)
      +d\bigl(W^{-1}{\rm Exp}(tD)[W^{T}]^{-1},{\rm Exp}(tD)\bigr).\nonumber
\end{align}
 Adding $- \|tV\|$ to every term in~\eqref{eq:BusSDP:p2} and using item (i) of  Lemma~\ref{lem:aux.main.theo} we obtain that 
$$
  \lim_{t\to\infty}   \bigl( d(W^{-1}{\rm Exp}(tD)[W^{T}]^{-1},Z Z^{T}) - \|tV\| \bigr)
  =\lim_{t\to\infty} \bigl( d({\rm Exp}(tD),Z Z^{T}) - \|tV\| \bigr).
$$
Hence, by \eqref{eq:defBFNDef} and \eqref{eq:BusSDP:p1}, and using the fact that
$\|tV\|=\|tD\|_{I}$, we have 
$$
   B_{Y, V}(X)= \lim_{t\to\infty}\!\bigl(  d({\rm exp}_{Y}(tV),X) - \|tV\| \bigr)
  =\lim_{t\to\infty}\!\bigl(  d({\rm Exp}(tD),Z Z^{T}) - \|tD\|_{I} \bigr)
  =B_{I,D}(Z Z^{T}).
$$
Thus,  since $Z Z^{T}$ commutes with $D$, Lemma~\ref{theo:main.commute} gives
\begin{equation} \label{eq:qfep}
  B_{Y, V}(X) = - \frac{1}{\|D\|_{I}} {\rm tr} \big( D {\rm Log}( Z Z^{T} ) \big).
\end{equation}
To conclude the proof, it remains to show that the right-hand side above coincides with the right-hand side of~\eqref{theo:bus.I.eq}. For that, first observe that
$$
D{\rm Log}\left( ZZ^T \right) = \begin{bmatrix}
\lambda_1 {\rm Log}\left( L_{n_1} L_{n_1}^T \right)  & 0 & \cdots & 0 \\
0 & \lambda_2 {\rm Log}\left( L_{n_2} L_{n_2}^T \right)  & \cdots & 0 \\
\vdots & \vdots & \ddots & \vdots \\
0 & 0 & \cdots & \lambda_k {\rm Log}\left( L_{n_k} L_{n_k}^T \right)
\end{bmatrix},
$$
and therefore
$$
  {\rm tr} \bigl(D\,{\rm Log}(Z Z^{T})\bigr)
  =\sum_{i=1}^{k}\lambda_i {\rm tr} \bigl({\rm Log}(L_{n_i}L_{n_i}^{T})\bigr)
  =\sum_{i=1}^{k}\lambda_i\ln \Bigl(\det\bigl(L_{n_i}L_{n_i}^{T}\bigr)\Bigr)
  =2\sum_{i=1}^{k}\lambda_i\sum_{j=\alpha_{i-1}+1}^{\alpha_i}\ln\bigl((L)_{jj}\bigr).
$$
Combining the last equality with  \eqref{eq:qfep} and taking into account that 
$
  \|D\|_{I}=  \left(  n_1\lambda_1^{2}+\cdots+n_k\lambda_k^{2}\right)^{1/2} ,
$
we obtain the desired equality.
\end{proof}
\begin{proof}[Gradient of Busemann functions on $\mathcal{P}(n) $   in Example~\ref{ex:BusSDP}:] 
Since $\bar{U} := L^{T} [X^{-1/2} Y^{1/2} U]^{T}$ is an orthogonal matrix, it follows from \eqref{eq:AffineInvariantMetric}, \eqref{eq:ExpSPD}, and \eqref{eq:LogSPD} that
\begin{align*}
\left\langle \log_{X} (\exp_{Y}(tV)) , K \right\rangle_X
& =   {\rm tr}\left(X^{-1/2}{\rm Log} \left( X^{-1/2}  Y^{1/2} {\rm Exp}\left( t  U D  U^{T} \right) Y^{1/2}  X^{-1/2} \right)  X^{-1/2} K \right) \\
& =   {\rm tr}\left({\rm Log} \left( \big[ X^{-1/2}  Y^{1/2} U \big] {\rm Exp}\left( t  D  \right) \bigl[    X^{-1/2} Y^{1/2} U \bigr]^{T} \right)  X^{-1/2} K X^{-1/2} \right) \\
& =   {\rm tr}\left({\rm Log} \left( \bar{U}^{T}  L^{-1}  {\rm Exp}\left( t  D  \right) \left( L^{T} \right)^{-1}  \bar{U}  \right)  X^{-1/2} K X^{-1/2} \right) \\
& =   {\rm tr}\left( \bar{U}^{T}  {\rm Log} \left(  L^{-1}  {\rm Exp}\left( t  D  \right) \left( L^{T} \right)^{-1}   \right)  \bar{U} X^{-1/2} K X^{-1/2} \right) \\
& =   {\rm tr}\left(   {\rm Log} \left(  L^{-1}  {\rm Exp}\left( t  D  \right) \left( L^{T} \right)^{-1}   \right)  \bar{U} X^{-1/2} K X^{-1/2} \bar{U}^{T} \right) 
\end{align*}
for all $K\in {\cal S}(n)$. Substituting $ K := - (1/ \|D\|_I ) \left(X^{1/2} \bar{U}^{T} D \bar{U} X^{1/2} \right) $ into the above equality and performing some algebraic simplifications, we obtain 
$$
 \left\langle \log_{X} (\exp_{Y}(tV)) , K \right\rangle_X 
 = - \frac{1}{\|D\|_I }  {\rm tr}\left(   {\rm Log} \left(  L^{-1}  {\rm Exp}\left( t  D  \right) \left( L^{T} \right)^{-1}   \right)  D\right).
$$
Applying this identity, together with further algebraic manipulations and the Cauchy–Schwarz inequality, we conclude that
\begin{align*}
& \lim_{t\to \infty} \frac{1}{ t }  \left\langle \log_{X} (\exp_{Y}(tV)) , K \right\rangle_X 
=   \frac{1}{\|D\|_I} \lim_{t\to \infty} \frac{1}{ t }  {\rm tr}\left( \left[ tD -  {\rm Log} \left(  L^{-1}  {\rm Exp}\left( t  D  \right) \left(  L^{T} \right)^{-1}   \right)    \right] D \right) 
  -  \|D\|_I \\
   &\leq   \frac{1}{\|D\|_I} \lim_{t\to \infty} \frac{1}{ t }  \left\|  tD  -  {\rm Log} \left(  L^{-1}  {\rm Exp}\left( t  D  \right) \left(  L^{T} \right)^{-1}   \right)    \right\|_I \left\| D \right\|_I
  -  \|D\|_I \\
     & =    \lim_{t\to \infty} \frac{1}{ t }  \left\| {\rm log}_I \left( {\rm Exp} \left(  tD \right) \right) -  {\rm log}_I \left(  L^{-1}  {\rm Exp}\left( t  D  \right) \left(  L^{T} \right)^{-1}   \right)    \right\|_I 
  -  \|D\|_I.
\end{align*}
Using \eqref{eq:log-norm-ineq} and \eqref{eq:LogSPD} together with item (ii) of Lemma~\ref{lem:aux.main.theo} we obtain that 
\begin{multline*}
 \lim_{t\to \infty} \frac{1}{ t }  \left\| {\rm log}_I \left( {\rm Exp} \left(  tD \right) \right) -  {\rm log}_I \left(  L^{-1}  {\rm Exp}\left( t  D  \right) \left(  L^{T} \right)^{-1}   \right)    \right\|_I  \leq \\  
  \lim_{t\to \infty} \frac{1}{ t }  d \left(  {\rm Exp} \left(  tD \right) ,  L^{-1}  {\rm Exp}\left( t  D  \right) \left(  L^{T} \right)^{-1}   \right)    =0.
\end{multline*}
Hence, it follows from two last inequalities that 
$$
 \lim_{t\to \infty} \frac{1}{ t }  \left\langle \log_{X} (\exp_{Y}(tV)) , K \right\rangle_X\leq    -  \|D\|_I.
$$
Therefore, applying Lemma~\ref{le:CharactBusFunc} and taking into account that $ \|V\| = \|D\|_I $ and $ \left\| {\rm grad} B_{Y, V}(X) \right\|_X = 1 = \left\| K \right\|_X $, we obtain
\begin{align*}
 \left\| {\rm grad} B_{Y, V}(X)  - K \right\|^2
& = \left\| {\rm grad} B_{Y, V}(X) \right\|^2_X - 2 \left\langle {\rm grad} B_{Y, V}(X), K \right\rangle_X + \left\| K \right\|^2_X \\
& = 2 - 2 \left\langle {\rm grad} B_{Y, V}(X), K \right\rangle_X  = 2 + \frac{2}{\| D \|_I} \lim_{t \to \infty} \frac{1}{t} \left\langle \log_{X}(\exp_{Y}(tV)), K \right\rangle_X \leq 0,
\end{align*}
which completes the proof of \eqref{eq:gradbf}, since $
  \|D\|_{I}=  \left(  n_1\lambda_1^{2}+\cdots+n_k\lambda_k^{2}\right)^{1/2} 
$.
\end{proof}
\end{appendixthm}
\bibliographystyle{habbrv}
\bibliography{OptConvHyperbolicBusemann}

@Article{Criscitiello2025,
  AUTHOR = {Criscitiello, Christopher and Kim, Jungbin},
  TITLE = {Horospherically Convex Optimization on {H}adamard Manifolds. {P}art {I}: Analysis and Algorithms},
  FJOURNAL = {arXiv preprint},
  JOURNAL = {arXiv},
  VOLUME = {2505},
  PAGES = {16970},
  YEAR = {2025},
  LANGUAGE = {English},
  EPRINT = {2505.16970},
  ARCHIVEPREFIX = {arXiv},
  PRIMARYCLASS = {math.OC},
  NOTE = {Available at \url{https://arxiv.org/abs/2505.16970}}
}

@Article{Goodwin2024,
  AUTHOR = {Goodwin, Alex and Lewis, Adrian S. and L{\'o}pez-Acedo, Gema and Nicolae, Adrian},
  TITLE = {A subgradient splitting algorithm for optimization on nonpositively curved metric spaces},
  FJOURNAL = {arXiv preprint},
  JOURNAL = {arXiv},
  VOLUME = {2412},
  PAGES = {06730},
  YEAR = {2024},
  LANGUAGE = {English},
  EPRINT = {2412.06730},
  ARCHIVEPREFIX = {arXiv},
  PRIMARYCLASS = {math.OC},
  NOTE = {Available at \url{https://arxiv.org/abs/2412.06730}}
}

@Article{Aragon2020,
 Author = {Arag{\'o}n Artacho, Francisco J. and Vuong, Phan T.},
 Title = {The boosted difference of convex functions algorithm for nonsmooth functions},
 FJournal = {SIAM Journal on Optimization},
 Journal = {SIAM J. Optim.},
 ISSN = {1052-6234},
 Volume = {30},
 Number = {1},
 Pages = {980--1006},
 Year = {2020},
 Language = {English},
 DOI = {10.1137/18M123339X},
 zbMATH = {7181880},
 Zbl = {1461.65119}
}

@Article{DeOliveira2020,
 Author = {de Oliveira, Welington},
 Title = {Sequential difference-of-convex programming},
 FJournal = {Journal of Optimization Theory and Applications},
 Journal = {J. Optim. Theory Appl.},
 ISSN = {0022-3239},
 Volume = {186},
 Number = {3},
 Pages = {936--959},
 Year = {2020},
 Language = {English},
 DOI = {10.1007/s10957-020-01721-x},
 zbMATH = {7246149},
 Zbl = {1450.90033}
}

@Article{PapaQuiro2013,
 Author = {Papa Quiroz, E. A.},
 Title = {An extension of the proximal point algorithm with {Bregman} distances on {Hadamard} manifolds},
 FJournal = {Journal of Global Optimization},
 Journal = {J. Glob. Optim.},
 ISSN = {0925-5001},
 Volume = {56},
 Number = {1},
 Pages = {43--59},
 Year = {2013},
 Language = {English},
 DOI = {10.1007/s10898-012-9996-y},
 zbMATH = {6176118},
 Zbl = {1295.90103}
}

@Article{PapaQuiro2009,
 Author = {Papa Quiroz, Erik A. and Oliveira, Paolo Roberto},
 Title = {Proximal point methods for quasiconvex and convex functions with {Bregman} distances on {Hadamard} manifolds},
 FJournal = {Journal of Convex Analysis},
 Journal = {J. Convex Anal.},
 ISSN = {0944-6532},
 Volume = {16},
 Number = {1},
 Pages = {49--69},
 Year = {2009},
 Language = {English},
 URL = {www.heldermann.de/JCA/JCA16/JCA161/jca16003.htm},
 zbMATH = {5586177},
 Zbl = {1176.90361}
}

@Article{Bregma1967,
 Author = {Bregman, L. M.},
 Title = {The relaxation method of finding the common point of convex sets and its application to the solution of problems in convex programming},
 FJournal = {Zhurnal Vychislitel'no{\u{\i}} Matematiki i Matematichesko{\u{\i}} Fiziki},
 Journal = {Zh. Vychisl. Mat. Mat. Fiz.},
 ISSN = {0044-4669},
 Volume = {7},
 Pages = {620--631},
 Year = {1967},
 Language = {Russian},
 zbMATH = {3296905},
 Zbl = {0186.23807}
}

@Article{Kiwiel1997,
 Author = {Kiwiel, Krzysztof C.},
 Title = {Proximal minimization methods with generalized {Bregman} functions},
 FJournal = {SIAM Journal on Control and Optimization},
 Journal = {SIAM J. Control Optim.},
 ISSN = {0363-0129},
 Volume = {35},
 Number = {4},
 Pages = {1142--1168},
 Year = {1997},
 Language = {English},
 DOI = {10.1137/S0363012995281742},
 zbMATH = {1116127},
 Zbl = {0890.65061}
}

@Article{Bauschke1997,
 Author = {Bauschke, Heinz H. and Borwein, Jonathan M.},
 Title = {Legendre functions and the method of random {Bregman} projections},
 FJournal = {Journal of Convex Analysis},
 Journal = {J. Convex Anal.},
 ISSN = {0944-6532},
 Volume = {4},
 Number = {1},
 Pages = {27--67},
 Year = {1997},
 Language = {English},
 zbMATH = {1046019},
 Zbl = {0894.49019}
}

@article{Bergmann2024,
    AUTHOR = {Bergmann, R. and  Ferreira, O.P. and  Santos, E. M. and Souza, J. C. O.},
     TITLE = {The Difference of Convex Algorithm on {H}adamard Manifolds},
   JOURNAL = {J. Optim. Theory Appl.},
  FJOURNAL = {Journal of Optimization Theory and Applications},
    VOLUME = {201},
      YEAR = {2024},
    NUMBER = {1},
     PAGES = {221--251},
      ISSN = {0022-3239},
     CODEN = {JOTABN},
   MRCLASS = {90C48 (49M30 58E25)},
  MRNUMBER = {1622188 (99f:90180)},
       DOI = {10.1007/s10957-024-02392-8},

}

@Article{Bento2015,
 Author = {Bento, G. C. and Ferreira, O. P. and Oliveira, P. R.},
 Title = {Proximal point method for a special class of nonconvex functions on {Hadamard} manifolds},
 FJournal = {Optimization},
 Journal = {Optimization},
 ISSN = {0233-1934},
 Volume = {64},
 Number = {2},
 Pages = {289--319},
 Year = {2015},
 Language = {English},
 DOI = {10.1080/02331934.2012.745531},
 zbMATH = {6417950},
 Zbl = {1310.49006}
}

@Article{Bento2010,
 Author = {Bento, G. C. and Ferreira, O. P. and Oliveira, P. R.},
 Title = {Local convergence of the proximal point method for a special class of nonconvex functions on {Hadamard} manifolds},
 FJournal = {Nonlinear Analysis. Theory, Methods \& Applications. Series A: Theory and Methods},
 Journal = {Nonlinear Anal., Theory Methods Appl., Ser. A, Theory Methods},
 ISSN = {0362-546X},
 Volume = {73},
 Number = {2},
 Pages = {564--572},
 Year = {2010},
 Language = {English},
 DOI = {10.1016/j.na.2010.03.057},
 zbMATH = {5720185},
 Zbl = {1282.49023}
}

@article{FerreiraLouzeiroPrudente2018subgrad,
    AUTHOR = {Ferreira, O. P. and Louzeiro, M. S. and Prudente, L. F.},
     TITLE = {Iteration-complexity of the subgradient method on {R}iemannian
              manifolds with lower bounded curvature},
   JOURNAL = {Optimization},
  FJOURNAL = {Optimization. A Journal of Mathematical Programming and
              Operations Research},
    VOLUME = {68},
      YEAR = {2019},
    NUMBER = {4},
     PAGES = {713--729},
      ISSN = {0233-1934},
   MRCLASS = {90C25 (90C60)},
  MRNUMBER = {3937053},
       DOI = {10.1080/02331934.2018.1542532},
      URL = {https://doi.org/10.1080/02331934.2018.1542532},
}

@Misc{Pham1986,
 Author = {Pham Dinh Tao and El Bernoussi, Souad},
 Title = {Algorithms for solving a class of nonconvex optimization problems. {Methods} of subgradients},
 Year = {1986},
 Language = {English},
 HowPublished = {Fermat days 85: {Mathematics} for optimization, {Toulouse}/{France} 1985, {North}-{Holland} {Math}. {Stud}. 129, 249-271 (1986).},
 zbMATH = {4041643},
 Zbl = {0638.90087}
}

@article{lewis2024,
  title={Horoballs and the subgradient method},
  author={Lewis, Adrian S and Lopez-Acedo, Genaro and Nicolae, Adriana},
  journal={arXiv preprint arXiv:2403.15749},
  year={2024}
}

@article{ghadimi2021,
  title={Hyperbolic {B}usemann learning with ideal prototypes},
  author={Ghadimi Atigh, Mina and Keller-Ressel, Martin and Mettes, Pascal},
  journal={Advances in Neural Information Processing Systems},
  volume={34},
  pages={103--115},
  year={2021}
}

@inproceedings{bonet2023-2,
  title={{S}liced-{W}asserstein on symmetric positive definite matrices for M/EEG signals},
  author={Bonet, Cl{\'e}ment and Mal{\'e}zieux, Beno{\i}t and Rakotomamonjy, Alain and Drumetz, Lucas and Moreau, Thomas and Kowalski, Matthieu and Courty, Nicolas},
  booktitle={International Conference on Machine Learning},
  pages={2777--2805},
  year={2023},
  organization={PMLR}
}

@inproceedings{bonet2023,
  title={Hyperbolic sliced-{W}asserstein via geodesic and horospherical projections},
  author={Bonet, Cl{\'e}ment and Chapel, Laetitia and Drumetz, Lucas and Courty, Nicolas},
  booktitle={Topological, Algebraic and Geometric Learning Workshops 2023},
  pages={334--370},
  year={2023},
  organization={PMLR}
  }

@article{hirai2023,
  title={Convex analysis on {H}adamard spaces and scaling problems},
  author={Hirai, Hiroshi},
  journal={Foundations of Computational Mathematics},
  pages={1--38},
  year={2023},
  publisher={Springer}
}

@Article{Bento2022,
 Author = {Bento, Glaydston de Carvalho and Cruz Neto, Jo{\~a}o Xavier and Melo, {\'I}talo Dowell Lira},
 Title = {Combinatorial convexity in {Hadamard} manifolds: existence for equilibrium problems},
 FJournal = {Journal of Optimization Theory and Applications},
 Journal = {J. Optim. Theory Appl.},
 ISSN = {0022-3239},
 Volume = {195},
 Number = {3},
 Pages = {1087--1105},
 Year = {2022},
 Language = {English},
 DOI = {10.1007/s10957-022-02112-0},
 zbMATH = {7619003}
}

@article {Bento2023,
    AUTHOR = {Bento, Glaydston de C. and Cruz Neto, Jo\~{a}o and Melo,
              \'{I}talo Dowell L.},
     TITLE = {Fenchel conjugate via {B}usemann function on {H}adamard
              manifolds},
   JOURNAL = {Appl. Math. Optim.},
  FJOURNAL = {Applied Mathematics and Optimization},
    VOLUME = {88},
      YEAR = {2023},
    NUMBER = {3},
     PAGES = {Paper No. 83, 29},
      ISSN = {0095-4616,1432-0606},
   MRCLASS = {90C25 (53C18 53C20 53C80 90C26)},
  MRNUMBER = {4646400},
       DOI = {10.1007/s00245-023-10060-y},
       URL = {https://doi.org/10.1007/s00245-023-10060-y},
}

@Article{BentoNeto2024,
 Author = {Bento, Glaydston de C. and Neto, Jo{\~a}o X. Cruz and Lopes, Jurandir O. and Melo, {\'I}talo D. L. and da Silva Filho, Pedro},
 Title = {A new approach about equilibrium problems via {Busemann} functions},
 FJournal = {Journal of Optimization Theory and Applications},
 Journal = {J. Optim. Theory Appl.},
 ISSN = {0022-3239},
 Volume = {200},
 Number = {1},
 Pages = {428--436},
 Year = {2024},
 Language = {English},
 DOI = {10.1007/s10957-023-02356-4},
 zbMATH = {7794721}
}

@book{Ballmann1995,
  author    = {Ballmann, Werner},
  title     = {Lectures on Spaces of Nonpositive Curvature},
  series    = {DMV Seminar},
  volume    = {25},
  publisher = {Birkh{\"a}user},
  address   = {Basel},
  year      = {1995},
}

@Book{Ballmann1985,
 Author = {Ballmann, Werner and Gromov, Mikhael and Schroeder, Viktor},
 Title = {Manifolds of nonpositive curvature},
 FSeries = {Progress in Mathematics},
 Series = {Prog. Math.},
 ISSN = {0743-1643},
 Volume = {61},
 Year = {1985},
 Publisher = {Birkh{\"a}user, Cham},
 Language = {English},
 zbMATH = {3949191},
 Zbl = {0591.53001}
}

@Book{Busemann1955,
 Author = {Busemann, H.},
 Title = {The geometry of geodesics},
 FSeries = {Pure and Applied Mathematics (Academic Press)},
 Series = {Pure Appl. Math., Academic Press},
 ISSN = {0079-8169},
 Volume = {6},
 Year = {1955},
 Publisher = {Academic Press, New York, NY},
 Language = {English},
 zbMATH = {3182580},
 Zbl = {0112.37002}
}

@Book{Bourbaki1995,
 Author = {Bourbaki, N.},
 Title = {General Topology},
 Year = {1955},
 Publisher = {Springer Berlin Heidelberg},
 Language = {English},
 doi = {10.1007/978-3-642-61701-0},
}

@Article{NetoFerreiraPerez2002,
 Author = {da Cruz Neto, J. X. and Ferreira, O. P. and Lucambio P{\'e}rez, L. R.},
 Title = {Contributions to the study of monotone vector fields},
 FJournal = {Acta Mathematica Hungarica},
 Journal = {Acta Math. Hung.},
 ISSN = {0236-5294},
 Volume = {94},
 Number = {4},
 Pages = {307--320},
 Year = {2002},
 Language = {English},
 DOI = {10.1023/A:1015643612729},
 zbMATH = {1822370},
 Zbl = {0997.53012}
}

@Article{ChenTeboulle1993,
 Author = {Chen, Gong and Teboulle, Marc},
 Title = {Convergence analysis of a proximal-like minimization algorithm using {Bregman} functions},
 FJournal = {SIAM Journal on Optimization},
 Journal = {SIAM J. Optim.},
 ISSN = {1052-6234},
 Volume = {3},
 Number = {3},
 Pages = {538--543},
 Year = {1993},
 Language = {English},
 DOI = {10.1137/0803026},
 zbMATH = {436624},
 Zbl = {0808.90103}
}

@Article{dePierroIusem1996,
 Author = {de Pierro, A. R. and Iusem, A. N.},
 Title = {A relaxed version of {Bregman}'s method for convex programming},
 FJournal = {Journal of Optimization Theory and Applications},
 Journal = {J. Optim. Theory Appl.},
 ISSN = {0022-3239},
 Volume = {51},
 Pages = {421--440},
 Year = {1986},
 Language = {English},
 DOI = {10.1007/BF00940283},
 zbMATH = {3930732},
 Zbl = {0581.90069}
}

@Article{NetoFerreiraLucambio2000,
 Author = {da Cruz Neto, J. X. and Ferreira, O. P. and Lucambio P{\'e}rez, L. R.},
 Title = {Monotone point-to-set vector fields},
 FJournal = {Balkan Journal of Geometry and its Applications (BJGA)},
 Journal = {Balkan J. Geom. Appl.},
 ISSN = {1224-2780},
 Volume = {5},
 Number = {1},
 Pages = {69--79},
 Year = {2000},
 Language = {English},
 zbMATH = {1509511},
 Zbl = {0978.90076}
}

@article {WangLiWangYao2015,
    AUTHOR = {Wang, Xiangmei and Li, Chong and Wang, Jinhua and Yao,
              Jen-Chih},
     TITLE = {Linear convergence of subgradient algorithm for convex
              feasibility on {R}iemannian manifolds},
   JOURNAL = {SIAM J. Optim.},
  FJOURNAL = {SIAM Journal on Optimization},
    VOLUME = {25},
      YEAR = {2015},
    NUMBER = {4},
     PAGES = {2334--2358},
      ISSN = {1052-6234},
   MRCLASS = {90C25 (53C21 65K05)},
  MRNUMBER = {3425379},
MRREVIEWER = {Wim van Ackooij},
       DOI = {10.1137/14099961X},
       URL = {http://dx.doi.org/10.1137/14099961X},
}

@article {LiMordukhovich2011,
    AUTHOR = {Li, Chong and Mordukhovich, Boris S. and Wang, Jinhua and Yao,
              Jen-Chih},
     TITLE = {Weak sharp minima on {R}iemannian manifolds},
   JOURNAL = {SIAM J. Optim.},
  FJOURNAL = {SIAM Journal on Optimization},
    VOLUME = {21},
      YEAR = {2011},
    NUMBER = {4},
     PAGES = {1523--1560},
      ISSN = {1052-6234,1095-7189},
   MRCLASS = {49J52 (90C31)},
  MRNUMBER = {2869507},
MRREVIEWER = {Dariusz\ Zagrodny},
       DOI = {10.1137/09075367X},
       URL = {https://doi.org/10.1137/09075367X},
}

@article {Todd2002,
 Author = {Nesterov, Yu. E. and Todd, M. J.},
 Title = {On the {Riemannian} geometry defined by self-concordant barriers and interior-point methods.},
 FJournal = {Foundations of Computational Mathematics},
 Journal = {Found. Comput. Math.},
 ISSN = {1615-3375},
 Volume = {2},
 Number = {4},
 Pages = {333--361},
 Year = {2002},
 Language = {English},
 DOI = {10.1007/s102080010032},
 URL = {hdl.handle.net/1813/9170},
 zbMATH = {1890491},
 Zbl = {1049.90127}
}

@article {Ferreira2002,
    AUTHOR = {Ferreira, O. P. and Oliveira, P. R.},
     TITLE = {Proximal point algorithm on {R}iemannian manifolds},
   JOURNAL = {Optimization},
  FJOURNAL = {Optimization. A Journal of Mathematical Programming and
              Operations Research},
    VOLUME = {51},
      YEAR = {2002},
    NUMBER = {2},
     PAGES = {257--270},
      ISSN = {0233-1934,1029-4945},
   MRCLASS = {49M30 (90C26 90C48)},
  MRNUMBER = {1928039},
MRREVIEWER = {R.\ Tichatschke},
       DOI = {10.1080/02331930290019413},
       URL = {https://doi.org/10.1080/02331930290019413},
}

@book {BridsonHaefliger1999,
    AUTHOR = {Bridson, Martin R. and Haefliger, Andr\'{e}},
     TITLE = {Metric spaces of non-positive curvature},
    SERIES = {Grundlehren der mathematischen Wissenschaften [Fundamental
              Principles of Mathematical Sciences]},
    VOLUME = {319},
 PUBLISHER = {Springer-Verlag, Berlin},
      YEAR = {1999},
     PAGES = {xxii+643},
      ISBN = {3-540-64324-9},
   MRCLASS = {53C23 (20F65 53C70 57M07)},
  MRNUMBER = {1744486},
MRREVIEWER = {Athanase\ Papadopoulos},
       DOI = {10.1007/978-3-662-12494-9},
       URL = {https://doi.org/10.1007/978-3-662-12494-9},
}

@book {Absil2008,
    AUTHOR = {Absil, P.-A. and Mahony, R. and Sepulchre, R.},
     TITLE = {Optimization algorithms on matrix manifolds},
      NOTE = {With a foreword by Paul Van Dooren},
 PUBLISHER = {Princeton University Press, Princeton, NJ},
      YEAR = {2008},
     PAGES = {xvi+224},
      ISBN = {978-0-691-13298-3},
   MRCLASS = {90-02 (58E17 90C30 90C52)},
  MRNUMBER = {2364186},
MRREVIEWER = {Anders\ Linn\'{e}r},
       DOI = {10.1515/9781400830244},
       URL = {https://doi.org/10.1515/9781400830244},
}

@book {Sakai1996,
    AUTHOR = {Sakai, Takashi},
     TITLE = {Riemannian geometry},
    SERIES = {Translations of Mathematical Monographs},
    VOLUME = {149},
      NOTE = {Translated from the 1992 Japanese original by the author},
 PUBLISHER = {American Mathematical Society, Providence, RI},
      YEAR = {1996},
     PAGES = {xiv+358},
      ISBN = {0-8218-0284-4},
   MRCLASS = {53-01 (53-02)},
  MRNUMBER = {1390760},
MRREVIEWER = {Conrad Plaut},
       DOI = {10.1090/mmono/149},
       URL = {https://doi.org/10.1090/mmono/149},
}

@book {Rapcsak1997,
AUTHOR = {Rapcs\'{a}k, Tam\'{a}s},
TITLE = {Smooth nonlinear optimization in {$\bold R^n$}},
SERIES = {Nonconvex Optimization and its Applications},
VOLUME = {19},
PUBLISHER = {Kluwer Academic Publishers, Dordrecht},
YEAR = {1997},
PAGES = {xiv+374},
ISBN = {0-7923-4680-7},
MRCLASS = {90C30 (65K05 90-02)},
MRNUMBER = {1480415},
MRREVIEWER = {Hubertus Th. Jongen},
DOI = {10.1007/978-1-4615-6357-0},
URL = {https://doi.org/10.1007/978-1-4615-6357-0},
}

@article {Kristaly2016,
    AUTHOR = {Krist\'{a}ly, Alexandru and Li, Chong and L\'{o}pez-Acedo, Genaro and
              Nicolae, Adriana},
     TITLE = {What do `convexities' imply on {H}adamard manifolds?},
   JOURNAL = {J. Optim. Theory Appl.},
  FJOURNAL = {Journal of Optimization Theory and Applications},
    VOLUME = {170},
      YEAR = {2016},
    NUMBER = {3},
     PAGES = {1068--1074},
      ISSN = {0022-3239},
   MRCLASS = {53C23 (53C21 53C24)},
  MRNUMBER = {3531813},
MRREVIEWER = {Jen-Chih Yao},
       DOI = {10.1007/s10957-015-0780-2},
       URL = {https://doi.org/10.1007/s10957-015-0780-2},
}

@book {Udriste1994,
    AUTHOR = {Udri\c{s}te, Constantin},
     TITLE = {Convex functions and optimization methods on {R}iemannian
              manifolds},
    SERIES = {Mathematics and its Applications},
    VOLUME = {297},
 PUBLISHER = {Kluwer Academic Publishers Group, Dordrecht},
      YEAR = {1994},
     PAGES = {xviii+348},
      ISBN = {0-7923-3002-1},
   MRCLASS = {49K27 (52A41 58C05 65K10 90C48)},
  MRNUMBER = {1326607},
MRREVIEWER = {Robert L. Foote},
       DOI = {10.1007/978-94-015-8390-9},
       URL = {https://doi.org/10.1007/978-94-015-8390-9},
}

@book {Boumal2020,
    AUTHOR = {Boumal, Nicolas},
     TITLE = {An introduction to optimization on smooth manifolds},
 PUBLISHER = {Cambridge University Press, Cambridge},
      YEAR = {2023},
     PAGES = {xviii+338},
      ISBN = {978-1-009-16617-1; 978-1-009-16615-7},
   MRCLASS = {90-01 (49J27 53B21 58Exx 90Cxx)},
  MRNUMBER = {4533407},
}

@book {Ratcliffe2019,
    AUTHOR = {Ratcliffe, John G.},
     TITLE = {Foundations of hyperbolic manifolds},
    SERIES = {Graduate Texts in Mathematics},
    VOLUME = {149},
   EDITION = {Third},
 PUBLISHER = {Springer, Cham},
      YEAR = {2019},
     PAGES = {800},
      ISBN = {978-3-030-31597-9; 978-3-030-31596-2},
   MRCLASS = {57M50 (20H10 30F40 57K32)},
  MRNUMBER = {4221225},
       DOI = {10.1007/978-3-030-31597-9},
       URL = {https://doi.org/10.1007/978-3-030-31597-9},
}

@book {BenedettiPetronio1992,
    AUTHOR = {Benedetti, Riccardo and Petronio, Carlo},
     TITLE = {Lectures on hyperbolic geometry},
    SERIES = {Universitext},
 PUBLISHER = {Springer-Verlag, Berlin},
      YEAR = {1992},
     PAGES = {xiv+330},
      ISBN = {3-540-55534-X},
   MRCLASS = {57M50 (30F40 30F60 51M10 57N10)},
  MRNUMBER = {1219310},
MRREVIEWER = {Colin C. Adams},
       DOI = {10.1007/978-3-642-58158-8},
       URL = {https://doi.org/10.1007/978-3-642-58158-8},
}

@book {Manfredo1992,
    AUTHOR = {do Carmo, Manfredo Perdig{\~a}o},
     TITLE = {Riemannian geometry},
    SERIES = {Mathematics: Theory \& Applications},
      NOTE = {Translated from the second Portuguese edition by Francis
              Flaherty},
 PUBLISHER = {Birkh\"auser Boston Inc.},
   ADDRESS = {Boston, MA},
      YEAR = {1992},
     PAGES = {xiv+300},
      ISBN = {0-8176-3490-8},
   MRCLASS = {53-01},
  MRNUMBER = {MR1138207 (92i:53001)},
MRREVIEWER = {Bang-yen Chen},
}

@misc{bao2024symcl,
title={Sym{CL}: Riemannian Contrastive Learning on the Symmetric Positive Definite Manifold for Visual Classification},
author={Yusheng Bao and Rui Wang and Tianyang Xu and Xiaojun Wu and Josef Kittler},
year={2024},
url={https://openreview.net/forum?id=SLUr06QUuw}
}

@INPROCEEDINGS{Tibermacine2024,
  author={Tibermacine, Ahmed and Tibermacine, Imad Eddine and Zouai, Meftah and Rabehi, Abdelaziz},
  booktitle={2024 International Conference on Telecommunications and Intelligent Systems (ICTIS)}, 
  title={{EEG} Classification Using Contrastive Learning and {R}iemannian Tangent Space Representations}, 
  year={2024},
  volume={},
  number={},
  pages={1-7},
  doi={10.1109/ICTIS62692.2024.10894645}
}

@article{manopt,
    author = {Boumal, N. and Mishra, B. and Absil, P.-A. and Sepulchre, R.},
    journal = {Journal of Machine Learning Research},
    title = {{M}anopt, a {M}atlab Toolbox for Optimization on Manifolds},
    year = {2014},
    number = {42},
    pages = {1455--1459},
    volume = {15},
    url = {https://www.manopt.org},
}
\end{document}